\documentclass[aap]{imsart}

\RequirePackage{amsthm,amsmath,amsfonts,amssymb}
\RequirePackage{subfig,mathrsfs,bbm,mathtools}
\RequirePackage[numbers]{natbib}
\RequirePackage[colorlinks,citecolor=blue,urlcolor=blue]{hyperref}
\RequirePackage{graphicx}

\startlocaldefs
\theoremstyle{plain}

\newtheorem{theorem}{Theorem}[section]
\newtheorem{lemma}[theorem]{Lemma}
\newtheorem{lem}[theorem]{Lemma}
\theoremstyle{remark}
\newtheorem{rem}[theorem]{Remark}
\newtheorem{example}[theorem]{Example}
\newtheorem{ass}[theorem]{Assumption}

\renewcommand{\d}{{\rm d}}

\newcommand{\id}{{\rm id}}

\newcommand{\norm}[1]{\left\Vert #1 \right\Vert}
\DeclareMathOperator*{\argmin}{argmin}
\endlocaldefs

\begin{document}

\begin{frontmatter}
\title{Wasserstein convergence rates of increasingly concentrating probability measures}
\runtitle{Convergence rates of increasingly concentrating probability measures}

\begin{aug}
\author[A]{\fnms{Mareike}~\snm{Hasenpflug}\ead[label=e1]{mareike.hasenpflug@uni-passau.de}},
\author[A]{\fnms{Daniel}~\snm{Rudolf}\ead[label=e2]{daniel.rudolf@uni-passau.de}}
\and
\author[B]{\fnms{Bj\"orn}~\snm{Sprungk}\ead[label=e3]{bjoern.sprungk@math.tu-freiberg.de}}
\address[A]{
	Faculty of Computer Science and Mathematics,
	University of Passau, Innstra{\ss}e 33, 94032 Passau, Germany\printead[presep={,\ }]{e1,e2}}

\address[B]{
	Technische Universit\"at Bergakademie Freiberg, Pr\"uferstr. 9, 09596 Freiberg, Germany\printead[presep={,\ }]{e3}}
\end{aug}

\begin{abstract}
For $\ell\colon \mathbb{R}^d \to [0,\infty)$ we
consider the sequence of probability measures $\left(\mu_n\right)_{n \in \mathbb{N}}$, where $\mu_n$ is determined by a density that is 
proportional to $\exp(-n\ell)$. 
We allow for infinitely many global minimal points of $\ell$, as long as they form a finite union of compact manifolds.
In this scenario, we 
show estimates for the $p$-Wasserstein convergence of $\left(\mu_n\right)_{n \in \mathbb{N}}$ to its limit measure.
Imposing regularity conditions we obtain a speed of convergence of  $n^{-1/(2p)}$ and adding a further technical assumption, we can improve this to a $p$-independent rate of $1/2$ for all orders $p\in\mathbb{N}$ of the Wasserstein distance. 
\end{abstract}



\end{frontmatter}
\section{Introduction}


We consider a sequence $(\mu_n)_{n\in\mathbb{N}}$ of increasingly concentrating probability measures on $\mathbb{R}^d$ given by
\[
\mu_n({\rm d}x) = \frac{1}{Z_n} \exp(-n \ell(x)) \mu_0({\rm d} x),
\qquad
Z_n := \int_{\mathbb{R}^d} \exp(-n \ell(x)) \mu_0({\rm d}x),
\]
where $n\in\mathbb{N}$, $\ell\colon \mathbb{R}^d \to [0,\infty)$ and 
$\mu_0$ denotes a reference measure on $\mathbb{R}^d$. It is intuitively plausible (and confirmed in \cite{Hwang}) that the limit measure, say $\mu$, of such a sequence 
is only supported
on the set of global minimal points of $\ell$. Allowing for infinitely many global minimal points that form a finite union of compact manifolds, we are driven by the question about explicit convergence rates and estimates for the $p$-Wasserstein distance (regarding the Euclidean metric and $p\in\mathbb{N}$) of $\mu_n$ to its limit measure $\mu$.


Let us provide the three main reasons for our interest 
in this 
question:
\begin{itemize} 
	\item 
	In \cite{Hwang} weak convergence of the sequence $\left(\mu_n\right)_{n \in \mathbb{N}}$ to explicitly stated limit measures is established and
	in our considered setting Wasserstein convergence and weak convergence are equivalent\footnote{This follows by virtue of \cite[Theorem 20.1.8]{Douc}.}. 
	Thus, explicit estimates and convergence rates of the Wasserstein distance allow us to quantify the
	weak convergence results of \cite{Hwang}. 
	\item  In Bayesian statistics, large sets of observational data or similarly small observational noise levels typically lead to strongly concentrated posterior distributions. From a modeling point of view this is a desirable scenario, however from a computational perspective it is also challenging, since most sampling methods suffer from high concentration. Knowledge about the rate at which the sequence of concentrating posterior distributions converges to its limit $\mu$ is important for the design and guarantees of cheap proxies of the posterior that mimic the same concentration behavior. 
	For example, in \cite{RuSp22,Sprungk} a Gaussian measure 
	given by the Laplace approximation of the posterior distribution is used to design concentration robust numerical integration methods.
	
	\item Recently in \cite{Debortoli} non-asymptotic estimates of the $1$-Wasserstein distance of $\mu_n$ to its limit measure (under appropriate assumptions) have been proven. 
	Those are used in different contexts, e.g., in delivering quantitative convergence of a stochastic gradient Langevin dynamics in a non-convex minimization setting or in studying the limit of a macrocanoncial model turning to a twisted microcanonical one. 
	Moreover, in \cite{Bras} similar results have been used to verify the convergence of a Langevin-simulated annealing algorithm with multiplicative noise.

\end{itemize}	
Those points indicate, apart from a theoretical point of view, that a quantitative Wasserstein convergence analysis based on concentrating probability measures applies in different mathematical fields.

Now we discuss the main assumptions and contributions. 
Regarding the function $\ell$ we require a manifold, an integrability, a tail, a smoothness, and a positive definiteness condition. 
We briefly comment on those conditions that are precisely formulated in Assumption~\ref{ass: reg_to_ell} below. The manifold requirement states that the set of minimal points of $\ell$ consists of a finite union of compact manifolds. The integrability assumption guarantees that the $p$-Wasserstein distance of interest is finite. Moreover, the manifold \textcolor{black}{condition} in combination with the tail condition implies that the set of global minimal points is non-empty. The positive definiteness assumption states that the Hessian of suitable projections of $\ell$ is positive definite at the minimal points of $\ell$.
Further, the smoothness, tail and positive definiteness condition allow us to apply Laplace's method for bounding several appearing integrals.
All imposed assumptions are closely related to the ones that are formulated and used in \cite{Hwang} in the treatment of the weak convergence statement of the most general case considered there. Therefore, we \textcolor{black}{have in particular} that the limit measure (w.r.t. weak convergence) of the sequence $(\mu_n)_{n\in\mathbb{N}}$ is well-defined. We denote it as $\mu$ and emphasize that its support is given through the set of minimal points.

We turn to the main results. In Theorem~\ref{Thm: Manifolds more general version} we prove that there exists a $K\in (0,\infty) $ such that 
\begin{equation}
\label{eq: n_1/2p}
W^p(\mu_n,\mu) \leq K n^{-1/(2p)}, \quad \forall n\in\mathbb{N},
\end{equation}
where $W^p$ denotes the $p$-Wasserstein distance for $p\in\mathbb{N}$. The $p$-dependence of the right hand-side is undesirable. \textcolor{black}{
	However, there is convincing numerical evidence, provided within the
	setting of Example~\ref{Ex: Slower convergence rate}, that indicates that a $p$ dependence can in general not be avoided under the assumptions of Theorem~\ref{Thm: Manifolds more general version}.
	Imposing an additional technical condition, that} ensures convergence from below of a sequence of proxies to the limit $\mu$, we show in Theorem~\ref{Thm: Manifolds convergence from below} that there is a number $K\in (0,\infty)$ such that
\begin{equation}
\label{eq: n_1/2}
W^p(\mu_n,\mu) \leq K n^{-1/2}, \quad \forall n\in\mathbb{N}.
\end{equation}
In general, the estimate of
\eqref{eq: n_1/2} cannot be improved w.r.t. the dependence on $n^{-1/2}$, since considering $\mu_n=\mathcal{N}(0,n^{-1})$, that is, $\mu_n$ is the normal distribution with mean zero and variance $n^{-1}$, and $\mu=\delta_0$ being the Dirac measure, one obtains $W^p(\mu_n,\mu) = \big(\frac{p!}{2^{p/2}(p/2)!}\big)^{1/p} n^{-1/2}$ for even $p$. For details we refer to Example~\ref{Ex: 1-d normal distribution} below.

In particular, the proof of Theorem~\ref{Thm: Manifolds convergence from below}, see \eqref{eq: n_1/2}, relies on a coupling construction between $\mu_n$ and $\mu$. Also the main part of the proof of Theorem~\ref{Thm: Manifolds more general version}, see \eqref{eq: n_1/2p}, is based on constructing a coupling between $\mu_n$ and a proxy of the limit distribution $\mu$ that respects a tubular neighbourhood structure of the manifold of the minimal points. The idea of the coupling is to shift the mass of $\mu_n$ that is in the aforementioned tubular neighbourhood to the manifold. The remaining mass outside of this set around the manifold is for sufficiently large $n$ very small, such that eventually it  can be almost arbitrarily used to fill the mass of the limit measure $\mu$ on the manifold.
For an illustration we refer to Figure~\ref{fig: coupling}. 
As the above mentioned normal distribution example shows, our coupling construction is optimal w.r.t. the convergence rate in $n$.
\begin{figure}[htb]
	\centering
	\subfloat[][]{\includegraphics[width=0.65\linewidth]{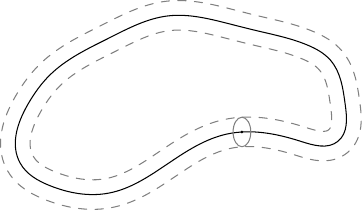}}%
	\qquad
	\subfloat[][]{\includegraphics[width=0.25\linewidth]{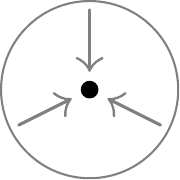}}%
	\caption{In (a) the black curve represents a $2$-dimensional manifold  
		given through the minimal points of $\ell$ surrounded by a $3$-dimensional tubular neighbourhood 
		determined by the dashed lines. The ellipse in (a) indicates a slice of the tubular neighbourhood and in (b) the same slice is drawn zoomed-in in two dimensions where it becomes a circle. The arrows indicate that the mass of $\mu_n$ within the tubular neighbourhood is, by the coupling, shifted to the manifold.}%
	\label{fig: coupling}
\end{figure}

Now we provide the outline of the paper: 
Section~\ref{Sec: Preliminaries} sets out our underlying setting, recalls the basic definitions about the $p$-Wasserstein distance and, in particular, in Section~\ref{Sec: preliminaries manifolds} we provide the required objects related to the manifold structure of the set of minimal points of $\ell$. 
Section~\ref{Sec: Main results} starts with the formulation of the assumptions and main results.
\textcolor{black}{We dedicate Section \ref{Sec: Diskussion of CB-condition} to the discussion of the additional technical assumption that we impose for obtaining $p$-independent convergence rates.}
 In Section~\ref{Sec: Examples} the developed theory is applied in illustrative scenarios and Section~\ref{Sec: Literature_review} contains a literature review with a special focus to quantitative results regarding sequences of concentrating probability measures.
After that, the proofs of the main theorems are delivered. 
Finally, some statements and proofs of auxiliary estimates required in the verification of the main contributions
are contained in the appendix.

\section{Preliminaries}\label{Sec: Preliminaries}
	\textcolor{black}{Let $\mathbb{N}$ be the positive natural numbers and $ \mathbb{N}_0 := \mathbb{N} \cup \{0\} $.}
For $d\in \mathbb{N}$ we consider $\mathbb{R}^d$ equipped with the Borel-$\sigma$-algebra $\mathcal{B}(\mathbb{R}^d)$\footnote{For $D\in \mathcal{B}(\mathbb{R}^d)$ the object $\mathcal{B}(D)$ denotes the trace $\sigma$-algebra of $\mathcal{B}(\mathbb{R}^d)$ in $D$.} induced by the Euclidean norm, and the Lebesgue measure $\lambda_d$. Let $\mu_0$ be a reference measure on $(\mathbb{R}^d,\mathcal{B}(\mathbb{R}^d))$ with strictly positive and infinitely often continuously differentiable Lebesgue density $\pi_0\colon \mathbb{R}^d \to (0,\infty)$, i.e., 
\begin{equation} 
\label{eq: prior}
\mu_0(\d x) = \pi_0(x)\, \lambda_d (\d x).
\end{equation}
Let $\ell: \mathbb{R}^d \to [0,\infty)$ be a measurable function that specifies a sequence of probability measures $(\mu_n)_{n\in\mathbb{N}}$ on $\left(\mathbb{R}^d, \mathcal{B}(\mathbb{R}^d)\right)$ 
with
\begin{equation}
\label{Eq: Definition of phi_n}
\mu_n(\d x) = \frac{1}{Z_n} \exp(-n\ell(x)) \,\mu_0(\d x), \qquad n \in \mathbb{N},
\end{equation}
where $Z_n$ denotes the normalizing constant. In terms of the Radon-Nicodym derivative this yields $
\frac{\d \mu_n}{\d \mu_0}(x) = \frac{\exp(-n\ell(x)) }{Z_n}$ for $x\in \mathbb{R}^d.$

\subsection{Wasserstein distance}
To measure the difference between probability measures, we use the Wasserstein distance.
For two probability measures $\mu$ and $\nu$ on $\left(\mathbb{R}^d, \mathcal{B}(\mathbb{R}^d)\right)$, we say that a probability measure 
$\gamma$ on $\left(\mathbb{R}^{2d}, \mathcal{B}(\mathbb{R}^{2d})\right)$ is a \emph{coupling} of $\mu$ and $\nu$ if for all sets $A \in \mathcal{B}(\mathbb{R}^d)$ we have
\begin{align*}
\gamma\left(A \times \mathbb{R}^d\right) &= \mu(A), \qquad
\gamma\left(\mathbb{R}^d \times A\right) = \nu(A).
\end{align*}
In other words the marginal distributions of $\gamma$ are $\mu$ and $\nu$.
The set of all couplings of $\mu$ and $\nu$ is denoted by $C(\mu, \nu)$. For $p\geq1$ the \emph{$p$-Wasserstein distance} of the probability measures $\mu$ and $\nu$ induced by the Euclidean norm $\norm{\cdot}$ is defined by
\[
W^p(\mu, \nu) := \left( \inf_{\gamma \in C(\mu, \nu)} \int_{\mathbb{R}^{2d}} \|x-y\|^p \, \gamma(\textup{d}x, \textup{d}y)  \right)^{1/p}.
\]
It is well known that $W^p$ is a metric on the set of probability measures on $\left(\mathbb{R}^d, \mathcal{B}(\mathbb{R}^d)\right)$ that satisfy a $p$-th moment condition, see e.g. \cite[Chapter~6]{Villani}.

From its definition it is clear that for  any coupling $\gamma\in C(\mu,\nu)$ holds
\[
W^p(\mu, \nu) \leq \left( \int_{\mathbb{R}^{2d}} \|x-y\|^p \, \gamma(\textup{d}x, \textup{d}y) \right)^{1/p}.
\]
For estimating the $p$-Wasserstein distance, we frequently use this simple fact and for an extensive treatment of further properties we refer to \cite[Chapter 6]{Villani}.

Sometimes it is more convenient to work with random variables. 
To this end, let $({\color{black}\widetilde{\Omega},\widetilde{\mathcal{F}}},\mathbb{P})$ be a sufficiently rich probability space which serves as domain of all sequel random variables. For example, for random variables $X,Y$ mapping to $\mathbb{R}^d$ writing $(X,Y)\sim \gamma$ with $\gamma\in C(\mu,\nu)$ means that the tuple $(X,Y)$ has distribution $\gamma$. 
In particular, then
\begin{equation}
\label{eq: bnd_expected_coupl_diff}
W^p(\mu,\nu) \leq (\mathbb{E} \Vert X-Y\Vert^p)^{1/p}.
\end{equation}
In this context, let us mention that we also use the notation $\mathbb{P}^X$ and $\mathbb{P}^{(X,Y)}$ for referring to the distribution of $X$ and the tuple $(X,Y)$, respectively. 
Now we turn to the conceptual objects regarding the considered manifold.

\subsection{Manifold framework}\label{Sec: preliminaries manifolds}


Let $k \in \mathbb{N}$, and let $M$ be a $k$-dimensional compact differentiable orientable Riemannian manifold, whose volume element we denote by $\Omega$.
The integration of exterior differentiable forms over $M$ induces an intrinsic measure $\mathcal{M}$ on $M$, i.e., 
\begin{equation}\label{Eq: Intrinsic measure}
\mathcal{M}(A) := \int_M \mathbbm{1}_A \Omega
\end{equation}
for all $A \in \mathcal{B}(M)$, where $\mathbbm{1}_A$ denotes the indicator function of $A$, and $\mathcal{B}(M)$ is the Borel-$\sigma$-algebra on $M$ induced by its topology.
For more details on this see \cite[Chapter~6]{Boothby}.


For $u \in M$, denote by $T_u(M)$ the tangent space to $M$ at $u$. 
If $M$ is a subset of $\mathbb{R}^d$ and is endowed with the subspace topology such that the identity map $\id_M: M \to \mathbb{R}^d$ is infinitely often continuously differentiable,
then the tangent space $T_u(M)$, which is $k$-dimensional, can be interpreted as a subset of $\mathbb{R}^d$ for every $u \in M$.
We denote the orthogonal complement of $T_u(M)$ in $\mathbb{R}^d$ by $T_u(M)^\perp$ and
by $B_r^s(x)\subset \mathbb{R}^s$ the Euclidean ball with radius $r>0$ around $x\in \mathbb{R}^s$.


Then, by \cite[Theorem 11.1]{classes}, there exists an $\varepsilon_0> 0$ such that for all $\varepsilon\in(0,\varepsilon_0)$ there is an open neighbourhood $N(\varepsilon) \subseteq \mathbb{R}^d$ of $M$ that is diffeomorph to
\[
E(\varepsilon) := \bigcup_{u \in M} \{u\} \times \left\{v \in T_u(M)^\perp \mid \|v\|  < \varepsilon \right\}
\]
via the map $
(u, v) \mapsto u+v.
$
The neighbourhood $N(\varepsilon)$ of $M$ is called \emph{tubular neighbourhood}.
If there exist $v^{(1)}(u), \ldots, v^{(d-k)}(u) \in T_u(M)^\perp$ that form an orthonormal basis of $T_u(M)^\perp$, 
such that $v^{(1)}, \dots, v^{(d-k)}$ depend smoothly on $u \in M$, then the map
\begin{align}\notag
S: B_\varepsilon^{d-k}(0) \times M &\to N(\varepsilon), \\
\label{al: def_S}
(t_1, \ldots, t_{d-k}, u) &\mapsto u + t_1 v^{(1)}(u) + \ldots + t_{d-k} v^{(d-k)}(u)
\end{align}
is a diffeomorphism, since $B_\varepsilon^{d-k}(0) \times M$ and $E(\varepsilon)$ are diffeomorphic.


Weyl provides in \cite{Weyl} a formula for the Lebesgue measure restricted to the tubular neighbourhood $N(\varepsilon)$, denoted by $\lambda_d\vert_{N(\varepsilon)}$, using a measure on $B_\varepsilon^{d-k}(0) \times M$ induced by the intrinsic measure $\mathcal{M}$.
More detailed, there exist infinitely often continuously differentiable functions $G_a^b(j) :M \to \mathbb{R}$, for $1 \leq j \leq d-k$ and $a,b \in \{1, \ldots, k\}$ defining the measure $Q$ on $B_\varepsilon^{d-k}(0) \times M$ by the density
\begin{align}\label{Eq: Density of Q}
\begin{aligned}
h:  B_\varepsilon^{d-k}(0) \times M & \to [0,\infty), \\
(t_1, \ldots, t_{d-k}, u) & \mapsto \det\left[\left(\delta_a^b + \sum_{j=1}^{d-k} t_j \left(G_a^b(j)\right)(u)\right)_{a,b \in \{1,\ldots, k\}}\right]
\end{aligned}
\end{align}
with respect to the product measure $\lambda_{d-k}\vert_{B_\varepsilon^{d-k}(0)} \otimes \mathcal{M}$, where $\delta_a^b$ is the Kronecker delta. 
(Note that, since $h(0,\dots,0,u) = 1$ for all $u \in M$ and $M$ is compact, the function $h$ is strictly positive if $\varepsilon$ is chosen small enough.)
Then, the pushforward measure $S_*(Q)$ of $Q$ under the map $S$ satisfies
\begin{equation}\label{Eq: Imagemeasure of Q under T}
S_*(Q) = \lambda_d\vert_{N(\varepsilon)},
\end{equation}
i.e., $Q\big(S^{-1}(A)\big) = S_*(Q)(A) = \lambda_d\vert_{N(\varepsilon)}(A)$ for all measurable sets $A \subseteq N(\varepsilon)$.
Note that
\begin{equation}\label{Eq: Normalizing constant of Q}
Q \left(B_\varepsilon^{d-k}(0) \times M\right) = S_*(Q)\left( S(B_\varepsilon^{d-k}(0) \times M) \right) = \lambda_d\big(N(\varepsilon)\big)
\end{equation}
and 
\begin{align}\label{Eq: Measureconversion on tubular neighbourhood}
\notag		&\int_{N(\varepsilon)} \exp(-n\, \ell(x)) \pi_0(x) \, \lambda_d(\d x)
=\int_{N(\varepsilon)} \exp(-n\,\ell(x)) \pi_0(x) \, S_*(Q)(\d x)\\
&\qquad = \int_{B_\varepsilon^{d-k}(0) \times M} \exp\big(-n\,\ell(S(t,u))\big) \pi_0\big(S(t,u)\big) \, Q(\d t, \d u) \\
\notag				&\qquad  =  \int_{M} \int_{B_\varepsilon^{d-k}(0)} \exp\big(-n\,\ell(S(t,u))\big) \pi_0\big(S(t,u)\big) h(t,u) 
\, \lambda_{d-k}(\d t) \, \mathcal{M}( \d u).
\end{align}
For $n \in \mathbb{N}$ and $u\in M$ define the inner integral of the former expression as
\begin{equation}\label{Eq: Definition of zeta}
\zeta^{(n)}(u) := \int_{B_\varepsilon^{d-k}(0)}   \exp\big(-n\,\ell(S(t,u))\big) \pi_0\big(S(t,u)\big) h(t,u) \, \lambda_{d-k}(\d t).
\end{equation}	

\begin{rem}
	Intuitively a tubular neighbourhood associates to each point of the enclosed manifold 
	a `slice' of the surrounding $\mathbb{R}^d$.
	In this light, the quantity $\zeta^{(n)}(u)$ can be interpreted as the `slicewise'-varying part of the integral 
	$\int_{N(\varepsilon)} \exp(-n \ell(x)) \pi_0(x) \lambda_d(\d x)$ that is orthogonal to the manifold $M$.
\end{rem}

We summarize the required assumptions that guarantee that all objects defined before exist.
\begin{ass}\label{A: Assumption on manifolds}
	Let $M \subset \mathbb{R}^d$ be a $k$-dimensional compact
	connected differentiable
	orientable Riemannian manifold which is endowed with the subspace topology
	such that the identity map $\id_M: M \to \mathbb{R}^d$ is an 
	infinitely often continuously differentiable function.
	For all $u \in M$ we assume that there exists an orthonormal basis $v^{(1)}(u), \ldots, v^{(d-k)}(u) \in T_u(M)^\perp$ of $T_u(M)^\perp$,
	where $v^{(1)}, \dots, v^{(d-k)}$ depend smoothly on $u$.
\end{ass}

In the following we consider a collection of finitely many manifolds $M_i$ with $i=1,\dots,m$, where each of them satisfies Assumption~\ref{A: Assumption on manifolds}. Therefore, we index all the objects which correspond to the manifold $M_i$ by $i$, that is, $M_i$ is $k_i$-dimensional with intrinsic measure $\mathcal{M}_i$, tubular neighbourhood $N_i(\varepsilon)$ for $\varepsilon>0$, orthornormal basis $v_i^{(1)}(u), \ldots, v_i^{(d-k_i)}(u)$ of $T_u(M_i)^\perp$ for $u\in M_i$, diffeomorphism $S_i$, measure $Q_i$ with density $h_i$, and $\zeta^{(n)}_i$ as stated in \eqref{Eq: Definition of zeta}.

\section{Wasserstein estimates}\label{Sec: Main results}

After Hwang \cite{Hwang} studied (depending on regularity assumption) the weak convergence of the sequence $(\mu_n)_{n \in \mathbb{N}}$ to its limit measure, 
we are aiming for the rate of this convergence w.r.t. to the $p$-Wasserstein distance. Clearly, the limit measure of $(\mu_n)_{n \in \mathbb{N}}$ depends on the form of 
the set of global minimal points of the function $\ell$ in the exponent of the density of $\mu_n$. 

\subsection{Main results and assumptions}

Similarly to the most general case treated in \cite{Hwang}, we impose the following regularity conditions.

\begin{ass}  \label{ass: reg_to_ell}
	For a given $p \in \mathbb{N}$ the function $\ell\colon \mathbb{R}^d\to [0,\infty)$ satisfies a
	\begin{itemize}
		\item[(M)] \emph{Manifold condition:} For the set $R_\ell:= \{x \in \mathbb{R}^d \mid \ell(x) = 0\}$ there exist $m\in\mathbb{N}$, $k_1,\dots,k_m\in\mathbb{N}$ and pairwise disjoint $k_i$-dimensional manifolds $M_i$ for $i=1,\dots,m$ satisfying Assumption \ref{A: Assumption on manifolds},
		such that $R_\ell = \bigcup_{i=1}^m M_i$.\\[-1.5ex]
		{\color{black}\item[(I)] \emph{Integrability condition:} 
		\label{A: Finiteness of integral}
		There exists $y_0 \in R_\ell$ such that  we have 
		\[
			\int_{\mathbb{R}^d} \|x-y_0\|^p \exp(-\ell(x)) \mu_0(\textup{d}x)<\infty.
		\]\\[-1.5ex]}
		{\color{black}\item[(T)] \emph{Tail condition:} 
		\label{A: Bounded away from zero}
		There exists an $\varepsilon > 0$ such that for all $i = 1, \ldots, m$ the tubular neighbourhoods $N_i(\varepsilon)$ (corresponding to $M_i$) are pairwise disjoint and for all $\delta \in (0, \varepsilon]$ holds
		\[
		\inf \Big\{\ell(x)  \mid x \in \mathbb{R}^d\setminus \bigcup_{i=1}^m N_i(\delta) \Big\} > 0.
		\]
		\\[-1.5ex]}
		\item[(S)]\emph{Smoothness condition:} 
				\label{A: Regularity of f}
		On a set $F\subseteq \mathbb{R}^d$ the function $\ell$ is 
		$2p + 2$-times continuously differentiable\footnote{All partial derivatives up to order $2p+2$ exist and are continuous.} and the tubular neighbourhoods $N_i(\varepsilon)$ are contained in the interior of $F$ for all $i = 1, \ldots, m$.\\[-1.5ex]
		\item[(P)] \emph{Positive definiteness condition:}
		For $i=1,\dots,m$ and $t=(t_1,\dots,t_{d-k_i})$ the Hessian $\displaystyle \frac{\partial^2 (
			\ell\circ S_i)}{\partial^2 t}(0,u)$ \textcolor{black}{of $t \mapsto (\ell \circ S_i)(t,u)$} is positive definite for all $u \in M_i$, where $S_i$ is the diffemorphism corresponding to $M_i$, see Section~\ref{Sec: preliminaries manifolds}.
	\end{itemize}
\end{ass}
Let us comment on the assumptions imposed on $\ell$. 
\begin{rem}
	\color{black}
	Condition~(I) implies for any $n\in\mathbb{N}$ that  
	\[  
	\int_{\mathbb{R}^d} \|x-y_0\|^p \exp(-n \,\ell(x)) \mu_0(\textup{d}x) \leq \int_{\mathbb{R}^d} \|x-y_0\|^p \exp(-\ell(x)) \mu_0(\textup{d}x)<\infty,
	\]	
which yields to
	\begin{align*}  
	Z_n &= \int_{\mathbb{R}^d} \exp(-n \,\ell(x)) \mu_0(\textup{d}x)\\
	&\leq \int_{\mathbb{R}^d \setminus B_1^{d}(y_0)} \|x-y_0\|^p \exp(-n \, \ell(x)) \mu_0(\textup{d}x)
	+ \int_{B_1^{d}(y_0)} \exp(-n \, \ell(x)) \mu_0(\textup{d}x) <\infty.
	\end{align*}
	This ensures the well-definedness of $\mu_n$ as well as the finiteness of Wasserstein distances of $\mu_n$ and any probability measure that satisfies a (similar) moment condition.
\end{rem}

\begin{rem}\label{R: Condition that implies (T)}
\color{black}
 By condition (M) we have that 
 $ R_\ell $ consists of finitely many pairwise disjoint compact sets and by the fact that the norm is a continuous function, the tail condition (T) is implied by asking for all $ \delta >0 $ that
 		\[
 		\inf \Big\{\ell(x)  \mid x \in \mathbb{R}^d\setminus \bigcup_{i=1}^m N_i(\delta) \Big\} > 0.
 		\]
 In general, the tail condition (T) prevents that there exists a sequence  $(x_k)_{k\in\mathbb{N}}$ of local minima
 of $\ell$ outside of 
 the union of the manifolds $M_i$
 such that $\ell(x_k) \to \min_{x\in\mathbb{R}^d} \ell(x)$ as $k\to \infty$. 
 In particular, this yields that the tails of $\mu_n$ become negligible compared to $\mu_n(N_j(\varepsilon))$ for $j=1,\ldots,m$ as $n\to \infty$.
\end{rem}

\begin{rem}
	Note that from the conditions (M) and 
	(T) we can conclude that $R_\ell$ is the set of all global minimal points. Moreover, condition (M) guarantees that this set consists of sufficiently well-behaved manifolds 
	such that the intrinsic measure and the tubular neighbourhood from Section~\ref{Sec: preliminaries manifolds} exist.
	The conditions (S), (T) and (P) are (also) required to apply Laplace's method, 
	for details see
	Theorem~\ref{Thm: Parametrized Laplace method} in Section~\ref{Sec: Proof of psi-zeta-convergence lemma}.
\end{rem}
\begin{rem}
	Note that in \eqref{eq: prior} we assumed for simplicity that the support of $\pi_0$ is the whole $\mathbb{R}^d$. 
	This is actually not required as long as 
	all minimal points of $\ell$ belonging to the support of $\pi_0$ lie within its interior.
\end{rem}	
Now we are able to formulate our Wasserstein distance estimates. 

\begin{theorem}\label{Thm: Manifolds more general version}
	Let $(\mu_n)_{n\in\mathbb{N}}$ be specified by \eqref{Eq: Definition of phi_n} with $\ell\colon \mathbb{R}^d\to [0,\infty)$ satisfying Assumption~\ref{ass: reg_to_ell}
	for a given $p \in \mathbb{N}$. With this, let
	$k :=  \max_{j = 1, \ldots, m}k_j $ and $R' := \bigcup_{i=1,\;\dim M_i = k}^m M_i$.
	Define $\phi\colon R' \to [0,\infty)$ by
	\begin{equation}
	\label{eq:phi}
	\phi (u) := \frac{ \sum_{i=1,\;M_i \subseteq R'}^m\; \pi_0 (S_i(0,u)) \det\left(\frac{\partial^2 (\ell \circ S_i)}{\partial^2 t}(0,u)\right)^{-1/2} 
		\mathbbm{1}_{M_i}(u)}
	{\sum_{i=1,\;M_i \subseteq R'}^m \int_{M_i} \pi_0(S_i(0,v))\det\left(\frac{\partial^2 (\ell \circ S_i)}{\partial^2 t}(0,v)\right)^{-1/2} \, \mathcal{M}_i(\textup{d}v)},
	\end{equation}
	and the measure $\mathcal{M}\colon \mathcal{B}(R')\to [0,\infty)$ by
	\[
	\mathcal{M}(\cdot) := \sum_{i=1,\;M_i \subseteq R'}^m \mathcal{M}_i( \cdot \cap M_i).
	\]
	Then, for the probability measure $\mu\colon \mathcal{B}(R') \to [0,1]$ given by
	\begin{equation}
	\label{eq:mu}
	\mu(\d u) := \phi(u) \, \mathcal{M}(\d u)
	\end{equation}
	there exists a constant $K\in(0,\infty)$ such that 
	\[
	W^p(\mu_n, \mu) \leq K \, n^{-1/(2p)}, \quad \forall n\in\mathbb{N}.
	\]
\end{theorem}

The proof is postponed to Section~\ref{Sec: proof_gen_thm}, but in the following remark we briefly comment on the strategy how to obtain the result. 
\begin{rem}  \label{rem: strategy_of_proof}
	For any $n\in\mathbb{N}$ define a sequence of (auxiliary) probability measure $(\nu_n)_{n\in\mathbb{N}}$ by setting
	\begin{align}  \label{al: nu_n}
	\nu_n(\d u) &:=\frac{1}{\sum_{j=1}^m \mu_n(N_j(\varepsilon)) } \sum_{i=1}^m \frac{\zeta^{(n)}_{i}(u)}{Z_n} \cdot \mathcal{M}_i(\d u),
	\end{align}
	where the function 
	\begin{equation}
	\label{eq: zeta_i}
	\zeta^{(n)}_i(u) = \int_{B_\varepsilon^{d-k_i}(0)}   \exp\big(-n\ell(S_i(t,u))\big) \pi_0\big(S_i(t,u)\big) h_i(t,u)
	\, \lambda_{d-k_i}(\textup{d}t), \quad u\in M_i,
	\end{equation}
	is given as in \eqref{Eq: Definition of zeta}.
	The probability measure $\nu_n$ is an approximation of the limit distribution $\mu$ which respects the structure introduced by the tubular neighbourhoods of $R_\ell$:
	\textcolor{black}{intuitively}, a tubular neighbourhood associates to each point of the enclosed manifold a `slice' of the surrounding $\mathbb{R}^d$.
	Therefore, $\nu_n$ can be interpreted as a projection of $\mu_n$, along this `slice'-structure, 
	onto the manifolds of minimal points.
	
	The idea is now to use the triangle inequality
	\[
	W^p(\mu_n,\mu) \leq W^p(\mu_n,\nu_n) + W^p(\nu_n,\mu),
	\]
	and to construct  \textcolor{black}{a suitable coupling from $C(\mu_n,\nu_n)$ to estimate $ W^p(\mu_n,\nu_n)$ and to bound $W^p(\nu_n,\mu)$ 
	by using Lemma~\ref{L: Coupling lemma for intermediate step} that exploits a total variation distance assessment.}
	Afterwards the resulting integrals get further estimated
	by using an adapted Laplace method, see also 
	\cite[Theorem 3.1]{Hwang}. 
\end{rem}


The order of convergence in Theorem \ref{Thm: Manifolds more general version} depends on 
the order
$p$ of the Wasserstein distance.
If we introduce additional assumptions, we get a rate of convergence that is independent of $p$.

\begin{theorem}\label{Thm: Manifolds convergence from below}
	Let $(\mu_n)_{n\in\mathbb{N}}$ be specified by \eqref{Eq: Definition of phi_n} with $\ell\colon \mathbb{R}^d\to [0,\infty)$ satisfying Assumption~\ref{ass: reg_to_ell} for a $p \in \mathbb{N}$. 
	For the manifolds $M_1,\dots,M_m$ assume that $\dim M_i = k$ for all $i\in\{1,\dots,m\}$.
	Define $\phi$ and $\mu$ as in \eqref{eq:phi} and \eqref{eq:mu}, respectively, and assume that
	\begin{equation}\label{Eq: Assumption convergence happens from below}
	\tag{CB}
	\phi(u) - \sum_{i=1}^m Z_n^{-1} \zeta_{i}^{(n)}(u) \mathbbm{1}_{M_i}(u) \geq 0, \quad \forall u\in R_\ell,\, n\in\mathbb{N},
	\end{equation}
	where the function $\zeta_{i}^{(n)}$ on $M_i$ is given as in \eqref{eq: zeta_i}.
	Then, there is a constant $K\in(0,\infty)$ such that
	\[
	W^p(\mu_n,\mu) \leq K n^{-1/2}, \qquad \forall n \in \mathbb{N}.
	\]
	
\end{theorem}


{\color{black}The proof of the result is postponed to Section~\ref{Sec: proof_est_spec}. 
We start with providing an intuition for the abbreviation \eqref{Eq: Assumption convergence happens from below}:
\begin{rem}
Compared to Theorem~\ref{Thm: Manifolds more general version} the crucial additional assumption of the former result
enforces a certain `convergence from below' of densities and is therefore abbreviated as \eqref{Eq: Assumption convergence happens from below}:
 Observe that in the previous theorem all requirements of Theorem~\ref{Thm: Manifolds more general version} are satisfied. Without taking \eqref{Eq: Assumption convergence happens from below} into account this already yields
\begin{equation}
\label{eq: limit}
\lim_{n \to \infty}  Z_n^{-1}\,\zeta_{i}^{(n)}(u) = \phi(u),
\end{equation}
for all $u \in M_i$ with $M_i \subseteq R'$ (see Lemma \ref{L: Convergence psi, zeta} below).     
Now \eqref{Eq: Assumption convergence happens from below} ensures that the convergence in \eqref{eq: limit} takes place from below.		
\end{rem}
An illuminating discussion about a special case follows:
\begin{rem}\label{R: Finitely many minima}
	We discuss Theorem~\ref{Thm: Manifolds more general version} and Theorem~\ref{Thm: Manifolds convergence from below} in the scenario where $\ell$ satisfies Assumption~\ref{ass: reg_to_ell} and has finitely many global minimal points, i.e., $R_\ell$ is a finite set. 
	Note that our setting covers such $\ell$, 
	since a single point in $\mathbb{R}^d$ can be interpreted as a $0$-dimensional manifold which always satisfies Assumption~\ref{A: Assumption on manifolds}.
	Moreover, the statements of the former theorems and Assumption~\ref{ass: reg_to_ell} simplify further in certain aspects. 
	The tubular neighbourhoods that satisfy (T)
	and (S)
	become balls of radius $\varepsilon>0$ centered at the points in $R_\ell$,
	and for (P) it is sufficient to ask for a positive definite Hessian matrix $H_\ell$ of $\ell$ evaluated at the global minimal points. 
	Note that the intrinsic measure of a singleton manifold $
	\{x\}$, with $x\in R_\ell$, is the Dirac measure $\delta_x$. The function $\phi$, as given in \eqref{eq:phi}, is defined on $R_\ell$ and therefore simplifies to the definition of weights. We set 
	\begin{equation}\label{Eq: Weights finitely many minima}
	w_x := \phi(x) = \frac{\det\left (H_\ell(x)\right)^{-1/2}\pi_0(x)}{\sum_{y \in R_\ell} \det\left (H_\ell(y)\right)^{-1/2}\pi_0(y)}, \qquad x\in R_\ell.
	\end{equation}
	Hence, the limit measure $\mu$ on $R_\ell$ is discrete and can be represented as
	\begin{equation}\label{Eq: Limit measure finitely many minima}
	\mu = \sum_{x \in R_\ell} w_x \delta_x.	
	\end{equation}
	Moreover, the assumption \eqref{Eq: Assumption convergence happens from below} of Theorem~\ref{Thm: Manifolds convergence from below} reduces to the existence of some $\varepsilon >0$ such that 
	\begin{equation}
	\tag{CB1}
	\label{eq: sqrtn_add_ass_finite_case}
	\mu_n\left(B_\varepsilon^d(x)\right) \leq w_x,
	\end{equation}
	for all $x \in R_\ell$ and all $n \in \mathbb{N}$.
	If the global minimal point $x$ of $\ell$ is unique, its weight simplifies to $w_x = 1$. 
	In 
	this case
	\eqref{eq: sqrtn_add_ass_finite_case} is always trivially satisfied.
	However, when $R_\ell$ is not a singleton, then \eqref{Eq: Assumption convergence happens from below} is not redundant.
	The fact, that in some form \eqref{Eq: Assumption convergence happens from below} is already required in such a simple setting also motivates further discussions.
\end{rem}
}

\subsection{\color{black}On the `convergence from below' assumption \eqref{Eq: Assumption convergence happens from below} of Theorem \ref{Thm: Manifolds convergence from below}}\label{Sec: Diskussion of CB-condition}

{\color{black}We provide our intuition about the additional requirement and state a sufficient condition.

 Intuitively, for $\mu_n$ the assumption \eqref{Eq: Assumption convergence happens from below} guarantees that  there is no more mass in the vicinity of a minimal point of $ \ell $ than this minimal point can attract in the limit.
In particular, this is satisfied if we have a certain uniformity of modes.
\begin{lemma}\label{L: Uniformity of modes}
	Without imposing \eqref{Eq: Assumption convergence happens from below} let all other assumptions of Theorem \ref{Thm: Manifolds convergence from below} be satisfied.
	Additionally, assume 
	there is a $ v \in R_\ell $ and an $ \widetilde{\varepsilon}>0 $ such that
	for all $ u \in R_\ell $ there exist invertible matrices $ A_u \in \mathbb{R}^{(d-k)\times(d-k)}$ that depend continuously on $u$, such that
	\begin{align}\label{Eq: Affinity conditions}
	\tag{CB2}
	\begin{split}
	\ell (S_i(t,u)) &= \ell(S_{j^*}(A_ut,v)),\\
	\pi_0(S_i(t,u)) &= \pi_0(S_{j^*}(A_ut, v)), \\
	h_i(t,u)&= h_{j^*}(A_ut,v),\qquad \qquad \forall u \in R_\ell,\; t \in A_u^{-1}\big(B_{\widetilde{\varepsilon}}^{d-k}(0)\big),
	\end{split}
	\end{align}
	where $ j^* $ is the index which satisfies $ v \in M_{j^*} $. 
	Then, condition \eqref{Eq: Assumption convergence happens from below} holds.
\end{lemma}
The proof of Lemma \ref{L: Uniformity of modes} is stated in Section \ref{Sec: Proof for uniformity of modes}. The requirement \eqref{Eq: Affinity conditions} can be interpreted as uniformity of the modes of $ \mu_n $ up to linear transformations. 
\begin{rem}\label{R: Uniformity of modes - Simpler formulations for special cases}
		We now describe \eqref{Eq: Affinity conditions} more explicitly in two different scenarios.
	\begin{enumerate}
		\item[(a)] If $ \ell $ has only finitely many minimal points, i.e., $ R_\ell = \{x_1, \ldots, x_m\} $, 
		then \eqref{Eq: Affinity conditions} means that there exists a minimal point, say $ x_1 $, 
		such that on a neighbourhood around each of the other minimal points, $ \ell $ and $ \pi_0 $ are affine linear transformations 
		of $ \ell $ and $ \pi_0 $ at $ x_1 $.
		More precisely
		\begin{equation}
		\tag{CB3}
		\label{Eq: Affinity condition for finitely many minima}
		\begin{split}
		\ell\vert_{A_{x_i}^{-1}\left(B_{\widetilde{\varepsilon}}^{d-k}(0)\right)+x_i}(\cdot) &= \ell\big(A_{x_i}(\cdot-x_i) + x_1\big), \\
		\pi_0\vert_{A_{x_i}^{-1}\left(B_{\widetilde{\varepsilon}}^{d-k}(0)\right)+x_i}(\cdot) &= \pi_0\big(A_{x_i}(\cdot-x_i) + x_1\big), \qquad \forall i \in \{2, \ldots,m\}.
		\end{split}
		\end{equation}
		\item[(b)] If, for any $ u \in R_\ell $, the matrix $A_u \in \mathbb{R}^{(d-k)\times(d-k)}$ coincides with the identity, 
		then \eqref{Eq: Affinity conditions} means that $ \ell \circ \left(\sum_{i=1}^{m} \mathbbm{1}_{M_i} S_i\right) $, $ \pi_0 \circ \left(\sum_{i=1}^{m} \mathbbm{1}_{M_i} S_i\right) $ and $ \left(\sum_{i=1}^{m} \mathbbm{1}_{M_i} h_i\right) $ are constant in $ u \in R_\ell $.
	\end{enumerate}
\end{rem}
\textcolor{black}{
Now we consider an example in the setting of Theorem~\ref{Thm: Manifolds more general version}, where a numerically assessment indicates  
that \eqref{Eq: Assumption convergence happens from below} is not satisfied. 
The experiments also suggest that the Wasserstein distances $W_p(\mu_n,\mu)$ decay like $n^{-r(p)}$ where in this case $1/p \geq r(p) \geq 1/(2p)$.
This indicates that \eqref{Eq: Assumption convergence happens from below} is required to ensure the $p$-independent convergence rate of $1/2$ within Theorem~\ref{Thm: Manifolds convergence from below}
}
\begin{example}\label{Ex: Slower convergence rate}
	\emph{(\textcolor{black}{Gaussian tails and two peaks}.)}
	Let $ \mu_0 = \lambda_1 $ be the $1$-dimensional Lebesgue measure and 
	consider the function
	\begin{align*}
	\ell: \mathbb{R} &\to [0, \infty),\qquad
	x \mapsto \mathbbm{1}_{\mathbb{R} \setminus [-1,1]}(x) (x-2)^2 + \mathbbm{1}_{[-1,1]} x^2 \exp(x),
	\end{align*}
	that determines $\mu_n$.
Observe that
	$
	R_\ell = \{0,2\}.
	$
	We verify that $ \ell $ satisfies the requirements of Theorem~\ref{Thm: Manifolds more general version} for finitely many minima as described in Remark~\ref{R: Finitely many minima}. 
	Since all moments of the normal distribution exist and $ \ell(x) = (x-2)^2 $ outside the interval $ [-1,1] $, condition (I) is satisfied.
	To verify (T), we choose an arbitrary $ \varepsilon \in (0,1) $, and note that it is sufficient to observe that $\ell$ is strictly monotonically decreasing on the sets $ (-\infty,2]\setminus [-1,1] $ and $ [-1,0] $,
	and strictly monotonically increasing on the sets $ [2,\infty) $ and $ [0,1] $.
	Clearly, also (S) is satisfied for the choice of $\varepsilon$.
	In particular, for second derivatives on the `balls' of radius $\varepsilon$ around the two minimal points we have
	\begin{align*}
	\ell''(x) &= \exp(x) \cdot \left(x^2 + 4x + 2\right),\quad |x| <\varepsilon,\\
	\ell''(x) &= 2 ,\quad |x-2|<\varepsilon.
	\end{align*} 
	Consequently $  
	H_\ell(0) = H_\ell(2) = 2  > 0,
	$
	such that (P) holds and we may apply Theorem \ref{Thm: Manifolds more general version}.
	The equality of the Hessian matrices at the two minimal points induces equal weights $ w_0 =w_2=\frac{1}{2} $, such that (by Remark~\ref{R: Finitely many minima}) the limit measure takes the form
	\[  
	\mu = \frac{1}{2} \delta_0 + \frac{1}{2} \delta_2.
	\]
	We are interested whether condition \eqref{Eq: Assumption convergence happens from below} is satisfied. Here it simplifies to \eqref{eq: sqrtn_add_ass_finite_case}, since we are in a finitely many minima case. For $z\in R_\ell = \{0,2\}$ consider  
	\[
	\varepsilon_*(n,z) = \sup\{ \varepsilon > 0 \mid \mu_n(B_\varepsilon^1(z)) \leq w_z \}, \qquad n \in \mathbb{N},
	\]
	and note that if \eqref{eq: sqrtn_add_ass_finite_case} holds, then
	$\lim_{n\to \infty} \varepsilon_*(n,z) >0$ for any $z\in\{0,2\}$.
	By approximating\footnote{For an initial $\varepsilon>0$ and fixed $ n \in \mathbb{N} $ we numerically evaluate $\mu_n(B_\varepsilon^1(z))$. If the resulting value is larger than $ w_z  = 1/2 $, we shrink $ \varepsilon $ by a factor of $8/10 $ and evaluate  $\mu_n(B_\varepsilon^1(z))  $  for the shrinked  $ \varepsilon $ until we find the first value $  \widehat{\varepsilon}_*=\widehat{\varepsilon}_*(n,z) $ such that $ \mu_n(B_{\widehat{\varepsilon}_*}^1(z)) \leq 1/2 $.
		For a sufficiently large initial $\varepsilon>0$ and under the assumption of exact integral computation we have by construction that $	\varepsilon_*(n,z)\in [\widehat{\varepsilon}_*(n,z),10/8\, \widehat{\varepsilon}_*(n,z)]$.} $\varepsilon_*(n,z)$ with $\widehat{\varepsilon}_*(n,z)$ we assess numerically whether the former limit behaviour is true. 
	Figure~\ref{fig: numerics1} (a) heavily supports the conjecture that
	\[
	\lim_{n \to \infty} {\varepsilon}_*(n,0) = 0 
	\quad \text{and} \quad
	\lim_{n \to \infty} {\varepsilon}_*(n,2) = 1,
	\]
	such that \eqref{eq: sqrtn_add_ass_finite_case} (and by equivalence \eqref{Eq: Assumption convergence happens from below}) cannot be satisfied.
\begin{figure}[htb]
	\centering
	\subfloat[][]{\includegraphics[width=0.475\linewidth]{
			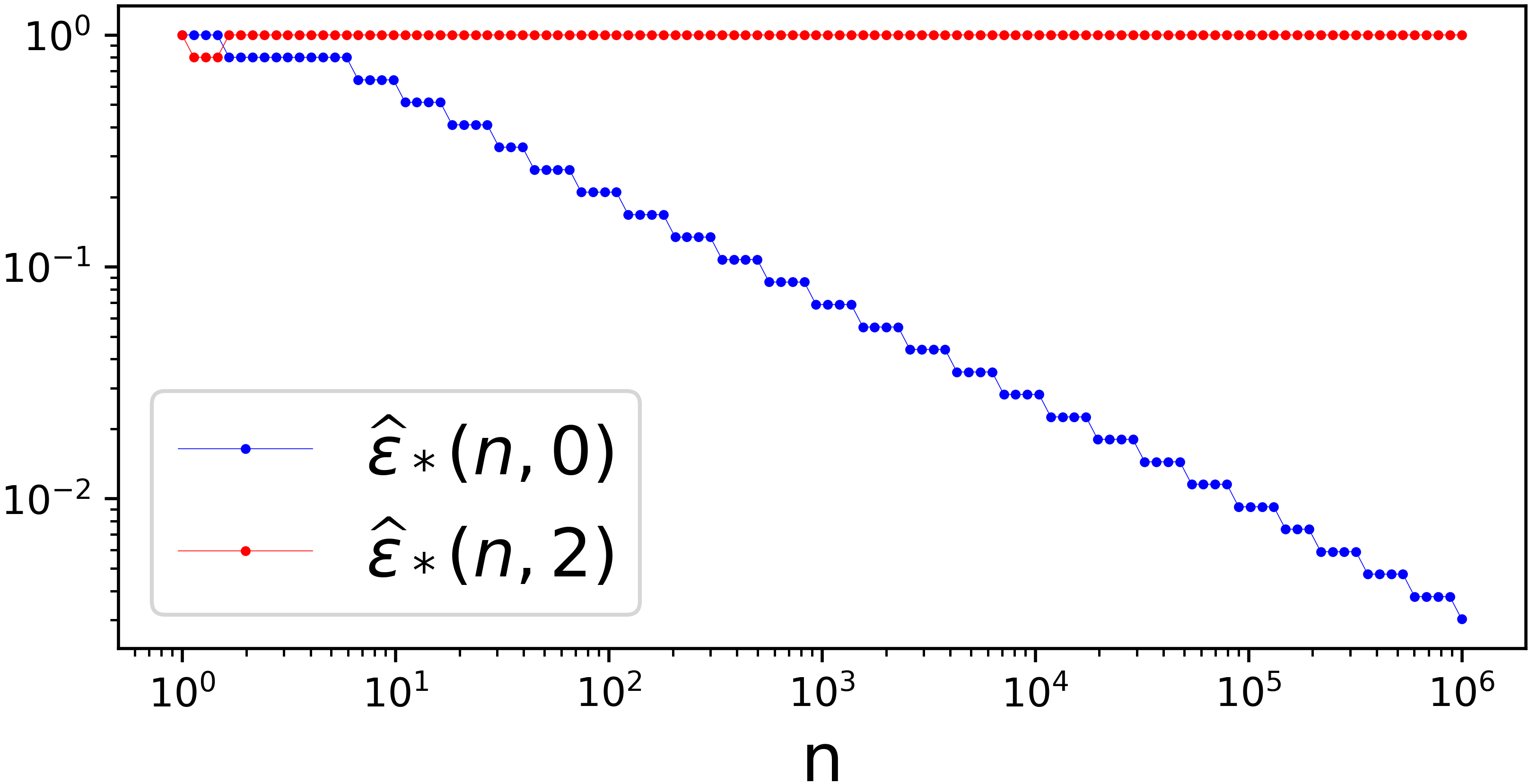}}
		\qquad
	\subfloat[][]{\includegraphics[width=0.475\linewidth]{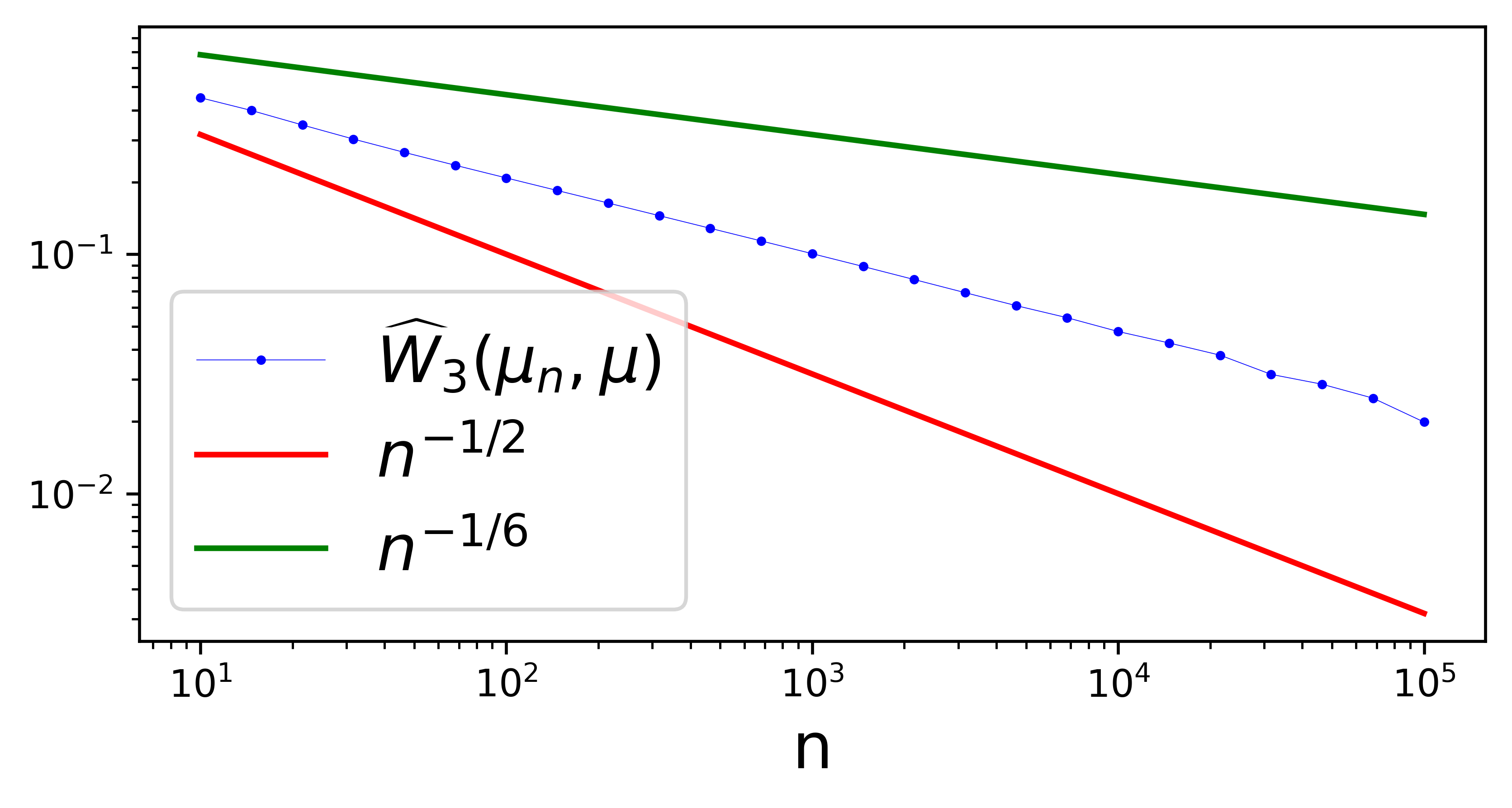}}%
			\caption{
		In (a) the proxy $n\mapsto \widehat{\varepsilon}_*(n,0)$ of $\varepsilon_*(n,0)$ is plotted in blue, whereas the proxy $n\mapsto \widehat{\varepsilon}_*(n,2)$ of $\varepsilon_*(n,2)$ is plotted in red. It suggests that $\lim_{n \to \infty} \varepsilon_*(n,0) = 0$, such that \eqref{Eq: Assumption convergence happens from below} is not satisfied. In (b) for $p=3$ the mapping $n\mapsto n^{-1/(2p)}$ is plotted in green, the mapping $n\mapsto n^{-1/2}$ is given by the red curve and the proxy $n \mapsto \widehat{W}_3(\mu_n,\mu)$, which behaves as $n\mapsto n^{-1/3}$, is presented in blue. 
	}
	\label{fig: numerics1}
\end{figure}

	Now we numerically assess $W^p(\mu_n,\mu)$. 
By \cite[Theorem~2.15]{Pflug_Pichler}, for $p\geq 1$ we have
\[
W^p(\mu_n,\mu) = \left( \int_0^1 \vert F_{\mu_n}^{-1}(r) - F_{\mu}^{-1}(r)  \vert^p {\rm d}r  \right)^{1/p},	\]
where $F^{-1}_{\mu_n}$, $F_{\mu}^{-1}$ denote the generalized inverses of the cumulative distribution functions $F_{\mu_n}, F_\mu$ of the probability measures $\mu_n,\mu$, respectively. 
By approximating $F_{\mu_n}^{-1}$ we assess a proxy $\widehat{W}^p(\mu_n,\mu)$ of $W^p(\mu_n,\mu)$ and plot $n\mapsto \widehat{W}^p(\mu_n,\mu)$ for fixed $p\geq1$. 
In Figure~\ref{fig: numerics1} (b) for $p=3$ and Figure~\ref{fig: numerics2} (a) for $p=18$ with $n\in\{1,\ldots, 10^5\} $ on a log-log-scale the corresponding mappings are presented. Let us mention that for other values of $p\geq 1$ we obtained qualitatively similar curves. The plots indicate that 
$W_p(\mu_n,\mu)$ behaves as $n^{-r(p)}$, with $r(p)$ being a $p$-dependent rate of convergence.
To take a closer look on this dependence we fit a regression model to approximate $r(p)$ by $\widehat{r}(p)$ based on numerically computed values of
$(p,n)\mapsto \widehat{W}_p(\mu_n,\mu)$, see Figure~\ref{fig: numerics2} (b). The visualization suggests that $p\leq r(p)^{-1} \leq 2p$ for all $p\geq1$ which clearly supports the aforementioned observation of the $p$-dependence.
\begin{figure}[htb]
	\centering
		\subfloat[][]{\includegraphics[width=0.475\linewidth]{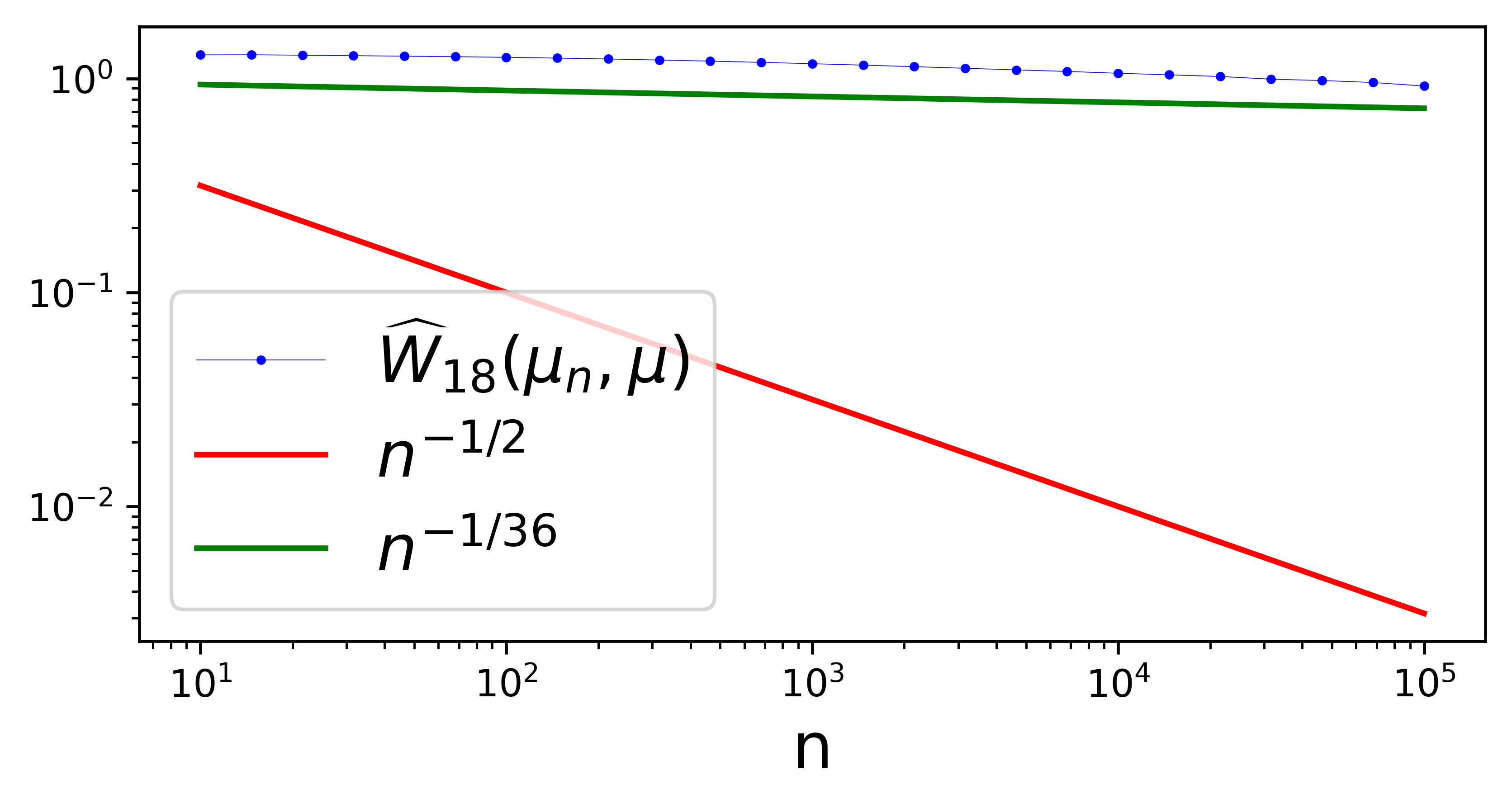}}%
		\qquad
		\subfloat[][]{\includegraphics[width=0.475\linewidth]{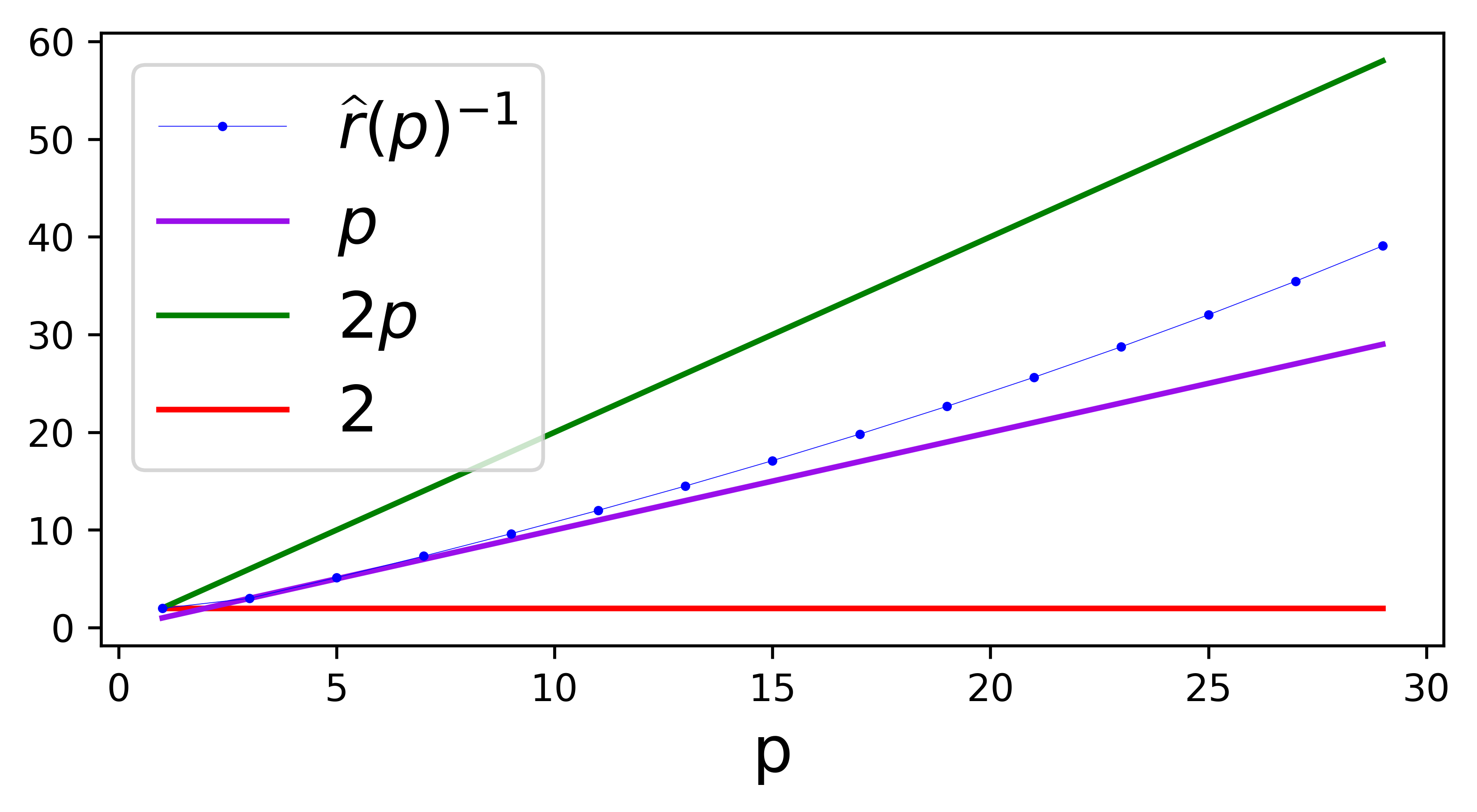}}
\caption{\color{black}
	In (a) for $p=18$ the mapping $n\mapsto n^{-1/(2p)}$ is plotted in green, the mapping $n\mapsto n^{-1/2}$ is given by the red curve and the proxy $n \mapsto \widehat{W}_{18}(\mu_n,\mu)$, which behaves as $n\mapsto n^{-1/21}$, is presented in blue. 
	In (b) the red curve is given by $p \mapsto 2$ (that corresponds to $n^{-1/2}$) as well as in green $p\mapsto 2p$ and in purple $p\mapsto p$ is plotted. The blue points present the proxies $\widehat{r}(p)$ of the convergence rate $r(p)$ depending on $p$. It suggest that $p\leq r(p) \leq 2p$ for all $p\geq 1$.   
	}%
	\label{fig: numerics2}
\end{figure}
\end{example}
}
{\color{black}
The previous example provides evidence that \eqref{Eq: Assumption convergence happens from below} 
is indeed required to obtain a $p$-independent convergence as stated in
Theorem~\ref{Thm: Manifolds convergence from below}. 
Therefore it seems appropriate to comment on the 
different strategies of proofs of both theorems.

In the proof of Theorem \ref{Thm: Manifolds more general version} we are only able to show that $W^p(\nu_n,\mu) \leq K n^{-1/(2p)}$ for some $K\in (0,\infty)$,
but $W^p(\mu_n, \nu_n ) \leq K' n^{-1/2}$ for some $K'\in (0,\infty)$,
i.e., the $p$-dependence within the order of convergence there comes from the use of the intermediate probability measures $(\nu_n)_{n\in\mathbb{N}}$.	
The additional assumption \eqref{Eq: Assumption convergence happens from below} in Theorem \ref{Thm: Manifolds convergence from below} allows us to directly construct a `good' coupling in $C(\mu,\mu_n)$,
which eventually leads to a convergence rate independent of $p$.}

\subsection{\color{black}Illustrative scenarios}\label{Sec: Examples}
{\color{black}We illustrate the applicability of the previous theorems at three examples. 
More specifically, we discuss the case of a $ 1 $-dimensional normal distribution, which establishes the optimality of the convergence rate in Theorem \ref{Thm: Manifolds convergence from below}.
We also show how our results can be used to derive convergence rates for suitable Gaussian (mixture) approximations of $\mu$ in the case of finitely many minima of $\ell$. 
Finally, we verify Assumption~\ref{ass: reg_to_ell} in a continuous manifold framework determined by volcano densities. }


\begin{example}\label{Ex: 1-d normal distribution}
	{\color{black}
	\emph{(Optimality of order of convergence.)}
	For $d=1$ let us consider $\mu_n = \mathcal{N}(0,n^{-1})$, i.e., $\mu_n$ is the $1$-dimensional normal distribution with mean zero and variance $n^{-1}$ for $n \in \mathbb{N}$.
	Thus, $\ell(x) = x^2/2$ and $\mu_0$ is the $1$-dimensional Lebesgue measure. 
	Taking Remark \ref{R: Finitely many minima} into account, it is easily seen that the assumptions of Theorem~\ref{Thm: Manifolds convergence from below} (with $m=1$ and $R_\ell=\{0\}$) are satisfied 
	for all $p \in \mathbb{N}$
	and $\mu=\delta_0$.

	In this particular situation, we observe that the upper bound of Theorem~\ref{Thm: Manifolds convergence from below} cannot be improved in general, since one can easily compute $W^p(\mu_n,\delta_0)$ for even $p$ explicitly. For this note that 
	$C(\mu_n,\delta_0)=\{\mu_n\otimes \delta_0\}$.
	%
	%
	%
	%
	Let $Z \sim \mathcal{N}(0,n^{-1})$ be a normally distributed real-valued random variable with mean $0$ and covariance $n^{-1}$.
	Then, the $p$th moment of $Z$ for even $p$ is given by $\frac{p!}{2^{p/2} (p/2)!}\left(\frac{1}{\sqrt{n}}\right)^p$ and
	therefore
	\begin{align*}
	W^p(\mu_n, \delta_0) &= \left(\int_{\mathbb{R}^2} |x-y|^p \, (\mu_n \otimes \delta_0)(\textup{d}x, \textup{d}y) \right)^{1/p}
	=\left( \mathbb{E}\left(Z^p\right) \right)^{1/p}\\
	&= \left(\frac{p!}{2^{p/2} (p/2)!} \right)^{1/p}\frac{1}{\sqrt{n}},
	\end{align*}
	which shows that the convergence rate of $1/2$ from Theorem~\ref{Thm: Manifolds convergence from below} cannot be improved.}
\end{example}

\begin{example}\emph{(Gaussian approximation.)}
	Suppose that $\ell$ satisfies Assumption~\ref{ass: reg_to_ell} and for the moment assume that $R_\ell = \{x^* \}$, i.e., there is only one global minimal point of $\ell$. In this setting a \emph{Gaussian approximation} of $\mu_n$ can be defined as 
	\begin{equation*}
	\mathcal{G}_{\mu_n} := \mathcal{N}\left( x^*, n^{-1} H_\ell(x^*)^{-1} \right),
	\end{equation*}
	that is, a multivariate normal distribution with mean $x^*$ and covariance matrix $n^{-1} H_\ell(x^*)^{-1}$, where $H_\ell(x^*)$ denotes the Hessian of $\ell$.
	By the triangle inequality of the $p$-Wasserstein distance we obtain
	\[
	W^p(\mu_n,\mathcal{G}_{\mu_n}) \leq W^p(\mu_n,\mu) + W^p(\mu,\mathcal{G}_{\mu_n}),
	\]
	so that by virtue of Remark~\ref{R: Finitely many minima}, in particular \eqref{eq: sqrtn_add_ass_finite_case} being trivially satisfied for both $\mu_n$ and $\mathcal{G}_{\mu_n}$, we have
	\begin{equation}
	\label{eq: Gaussian_proxy}
	W^p(\mu_n,\mathcal{G}_{\mu_n}) \leq K n^{-1/2}, \qquad \forall n \in \mathbb{N},
	\end{equation}
	for some $K\in(0,\infty)$.
	In this sense $\mathcal{G}_{\mu_n}$ serves as a simple proxy for $\mu_n$ that is
	conceptually close to the \emph{Laplace approximation}, which, for $\mu_n$ as in \eqref{Eq: Definition of phi_n}, is given by the Gaussian measure $\mathcal{N}\left( x_n, n^{-1} C_n^{-1} \right)$ where
	\[
	x_n = \argmin_{x\in\mathbb{R}^d} \Big(\ell(x) - \frac 1n \log \pi_0(x)\Big )
	\]
	and $C_n$ denotes the Hessian of $\ell - \frac 1n \log \pi_0$ at $x_n$.
	\textcolor{black}{Convergence guarantees 
		and error estimates of the Laplace approximation in the context of Bayesian inverse problems have been studied at least in \cite{helin2022non,Sprungk} in terms of the total variation distance and Hellinger metric. In particular, estimates as in \eqref{eq: Gaussian_proxy} are fundamental to analyze sampling schemes based on Gaussian approximations, which do not deteriorate for increasing concentration of $\mu_n$, see \cite{RuSp22,Sprungk}.}
	
	Now we consider the case of multiple but finitely many minimal points of $\ell$, i.e., $R_\ell$ is a finite set.
	For simplicity, we assume that for all $x\in R_\ell$ and $n\in\mathbb N$ the measures $\mu_n$ satisfy \eqref{eq: sqrtn_add_ass_finite_case}.
	As before we seek for a suitable ``Gaussian'' approximation $\mathcal{G}_{\mu_n}$ of $\mu_n$ such that \eqref{eq: Gaussian_proxy} holds. 
	Due \textcolor{black}{to Theorem \ref{Thm: Manifolds convergence from below}} and the triangle inequality this is guaranteed if
	\begin{equation} 
	\label{eq: suff_conv}
	W^p(\mathcal{G}_{\mu_n},\mu) \leq K n^{-1/2}, \qquad \forall n \in \mathbb{N},
	\end{equation}
	for some $K\in (0,\infty)$. 
	In the following we suggest two choices for a Gaussian approximation of $\mu_n$ that generalize the case of a single minimal point and indeed satisfy \eqref{eq: suff_conv}, and thus \eqref{eq: Gaussian_proxy}:
	
	\emph{1. Mixture of Gaussians:}
	For $n\in\mathbb{N}$ and $z \in R_\ell$ define
	$
	\nu_z^{(n)} := 	\mathcal{N}\left( z, n^{-1} H_\ell(z)^{-1} \right)
	$
	and set
	\[
	\mathcal{G}_{\mu_n} := \sum_{z \in R_\ell} w_z \ \nu_z^{(n)}, 
	\]
	where the weights $w_z$ for $z\in R_\ell$ are given as in \eqref{Eq: Weights finitely many minima}.
	
	Considering the discussion in Remark \ref{R: Finitely many minima} we obtain that for some $K_z \in (0, \infty)$ Theorem \ref{Thm: Manifolds convergence from below} yields 
	$W^p\left(\nu_z^{(n)}, \delta_z\right) \leq K_z n^{-1/2}$. 
	We define the coupling $\eta_n\in C(\mathcal{G}_{\mu_n},\mu)$ by
	\[
	\eta_n({\rm d}x,{\rm d}y) := \sum_{z \in R_\ell} w_z \left( \nu_z^{(n)}({\rm d}x)\otimes \delta_z({\rm d}y) \right),
	\]
	and by employing the fact that $C(\nu_z^{(n)}, \delta_z) 
	= \{ \nu_z^{(n)} \otimes \delta_z \}$ we have
	\begin{align*}
	\left( W^p(\mathcal{G}_{\mu_n}, \mu) \right)^p  
	&\leq \int_{\mathbb{R}^d \times \mathbb{R}^d} \norm{x -y}^p \eta_n( \d x, \d y)\\
	&= \sum_{z \in R_\ell} w_z \int_{\mathbb{R}^d \times \mathbb{R}^d} \norm{x -y}^p  
	\left(\nu_z^{(n)} \otimes \delta_z\right)( \d x, \d y)\\
	&  = \sum_{z \in R_\ell} w_z W^p\left(\nu_z^{(n)}, \delta_z\right)^p
	\leq \left( \sum_{z \in R_\ell} w_z K_z^p \right) n^{-p/2}.
	\end{align*}

	\emph{2. Maximum of Gaussians:}
	For the second generalization of the Gaussian approximation we take the maximum of the normal density functions 
	that have been used in the convex combination of the first suggestion. 
	For simplicity we assume $\mu_0 = \lambda_d$.
	We define
	\[
	\alpha(x) := \frac{1}{2} \min_{z \in R_\ell}\{(x-z)^T H_\ell(z) (x-z)\}, \qquad x \in \mathbb{R}^d,
	\]	
	and set, with $\widehat{Z}_n$ being the corresponding normalizing constant, 
	\[
	\mathcal{G}_{\mu_n}({\rm d}x) := \frac{1}{\widehat{Z}_n} \exp(-n \alpha(x)) \lambda_d(\d x), \qquad n \in \mathbb{N}.	
	\]
	
	We now check the assumptions of Theorem \ref{Thm: Manifolds convergence from below} as stated in Remark~\ref{R: Finitely many minima}.
	Note that we have $R_\alpha = R_\ell$, and for all $z \in R_\alpha$ holds $\alpha(x) = \frac{1}{2} (x-z)^T H_\ell(z) (x-z)$ on a neighbourhood $U_z$ of $z$.
	This implies $H_\alpha(z) = H_\ell(z)$ for all $z \in R_\alpha$.
	Therefore, conditions (S) and (P) are satisfied. 
	Since $H_\ell(z)$ is positive definite for $z \in R_\ell$, there exists a positive definite matrix $H_\ell(z)^{1/2}$ such that $H_\ell(z)^{1/2} H_\ell(z)^{1/2} = H_\ell(z)$.
	\textcolor{black}{Observe that \eqref{Eq: Affinity condition for finitely many minima} holds for $ A_{z} = H_\ell(\widetilde{z})^{-1/2}H_\ell(z)^{1/2} $, $ z \in R_\ell $, where $ \widetilde{z}\in R_\ell $ is fixed. }
	\textcolor{black}{Therefore \eqref{eq: sqrtn_add_ass_finite_case} is satisfied, see Remark \ref{R: Uniformity of modes - Simpler formulations for special cases}\,(a).}
	Observe that $x\mapsto \frac{1}{2} (x-z)^T H_\ell(z) (x-z)$ fulfils condition (T) for all $z \in R_\ell$.
	Hence, also $\alpha$ satisfies condition (T).
	Finally, note that $ \exp(-n \alpha(x)) \leq \sum_{z \in R_\ell} \exp \left( - \frac{1}{2} n (x-z)^T H_\ell(z) (x-z) \right)$ for all $x \in \mathbb{R}^d$.
	This implies condition (I), such that Theorem~\ref{Thm: Manifolds convergence from below} yields for some $K \in (0, \infty)$
	\[
	W^p\left( \mathcal{G}_{\mu_n}, \mu\right) \leq K n^{-1/2}, \qquad \forall \, n \in \mathbb{N}.
	\]
	
\end{example}


\begin{example}\emph{(Volcano densities.)}
	We consider the function $\ell \colon \mathbb{R}^d \to [0,\infty)$ given by
	$
	\ell(x) \coloneqq \frac{1}{2}\|x\|^2 - \|x\| + \frac{1}{2}
	$
	and set $\mu_0$ to be the Lebesgue measure on $\mathbb{R}^d$.
	We check the conditions (M), (I), (T), (S) and (P) of Assumption~\ref{ass: reg_to_ell} for arbitrary $p \in \mathbb{N}$.
	
	Note that \[R_\ell = \mathbb{S}^{d-1} := \{ x\in\mathbb{R}^d \colon \Vert x\Vert =1\},\] which is a 
	connected
	compact differentiable and orientable $(d-1)$-dimensional Riemannian
	manifold. The normal vector $v^{(1)}(u)=v(u)$ at $u \in \mathbb{S}^{d-1}$, where $u$ is interpreted as a vector in $\mathbb{R}^d$, is given by $u$.
	In particular, this means that the normal vector $v(u)$ is smooth in $u$ and (M) holds with $m=1$.

	We can choose $r > 0$ 
	sufficiently large, such that
	\begin{equation}\label{Eq: Lower bound on ell}
	\ell(x) = \frac{1}{2}\|x\|^2 - \|x\| + \frac{1}{2} 
	=
	\frac 12 \left(\|x\| -1 \right)^2
	\geq \frac{1}{4}\|x\|^2 
	\end{equation}
	for all $x \in \mathbb{R}^d$ with $\|x\| \geq r$. 
	Hence, condition (I) is satisfied for any finite $p \in \mathbb{N}$ due to the lower bound in \eqref{Eq: Lower bound on ell} and the fact that $d$-dimensional normal distributions possess finite moments of any order.

	Moreover, the tubular neighbourhood of $\mathbb{S}^{d-1}$ for $\varepsilon\in (0,1)$ is given as
	\[
	N_1(\varepsilon) = N(\varepsilon) = \{ x\in\mathbb{R}^d \colon 1-\varepsilon \leq \Vert x\Vert \leq 1+\varepsilon \}.
	\]
	Thus, by \eqref{Eq: Lower bound on ell} we have $\inf\{\ell(x) \colon | \norm{x} - 1 | \geq \delta \} = \frac 12 \delta^2 > 0$ for all $\delta > 0$,
	which shows that condition (T) is satisfied by Remark \ref{R: Condition that implies (T)}.
	
	Since $N(\varepsilon) \subseteq \mathbb{R}^d \setminus \{0\}$, where $\ell$ is infinitely often continuously differentiable,
	also (S) is satisfied for all $p \in \mathbb{N}$.
	
	Recall that the normal vectors can be represented as $v(u)=u$ for $u\in\mathbb{S}^{d-1}$.
	With this we obtain for any $\varepsilon \in(0,1)$ that
	\begin{align*}
	S: B^{1}_\varepsilon(0) \times \mathbb{S}^{d-1} & \to N(\varepsilon),
	\quad \text{with} \quad
	(t, u) \mapsto (1+t)u.
	\end{align*}
	The first two derivatives of $\ell \circ S$ in the first component are given by
	\begin{align*}
	\frac{\partial (\ell \circ S)}{\partial t} (t, u) &= \sum_{i=1}^d (u_i + tu_i) u_i 
	- \frac{\sum_{i=1}^d (u_i + tu_i) u_i}{\left(\sum_{i=1}^d (u_i + tu_i)^2 \right)^{1/2}}, \\
	\frac{\partial^2 (\ell \circ S)}{\partial^2 t} (t, u) &= \sum_{i=1}^d u_i^2 
	+ \frac{\left(\sum_{i=1}^d (u_i + tu_i) u_i\right)^2}{\left(\sum_{i=1}^d (u_i + tu_i)^2 \right)^{3/2}}  
	-\frac{\sum_{i=1}^d u_i^2}{\left(\sum_{i=1}^d (u_i + tu_i)^2 \right)^{1/2}},
	\end{align*}
	such that 
	\begin{align*}
	\frac{\partial^2 (\ell \circ S)}{\partial^2 t} (0, u) = \|u\|^2=1 > 0
	\end{align*}
	for all $u \in \mathbb{S}^{d-1}$, i.e., (P) holds.
	Hence, $\ell$ satisfies Assumption~\ref{ass: reg_to_ell} and Theorem~\ref{Thm: Manifolds more general version} is applicable.
	
	The density $\phi$ of the limit measure $\mu$ with respect to the intrinsic measure $\mathcal{M}$ on $\mathbb{S}^{d-1}$ satisfies
	\begin{align*}
		\color{black}
	\phi(u) = \frac{ \det\Big(\frac{\partial^2 (
			\ell \circ S)}{\partial^2 t} (0, u)\Big)^{-1/2}}{\int_{\mathbb{S}^{d-1}} \det\Big( \frac{\partial^2 (\ell \circ S)}{\partial^2 t} (0, v)\Big)^{-1/2}
		\, \mathcal{M}(\textup{d}v)}
	= \frac{1}{\mathcal{M}(\mathbb{S}^{d-1})}, \qquad u\in\mathbb{S}^{d-1},
	\end{align*}
	which implies $\mu({\rm d}u) = \mathcal{M}({\rm d}u)/\mathcal{M}(\mathbb{S}^{d-1})$.
	
	We now want to verify the additional assumption required for applying Theorem \ref{Thm: Manifolds convergence from below}.
{\color{black}	Observe that we are in the scenario described in Remark \ref{R: Uniformity of modes - Simpler formulations for special cases}\,(b), 
	as $ \pi_0 \equiv 1 $, $ \ell \circ S (t,u) =  \frac{1}{2}(1+t)^2 - (1+t) + \frac{1}{2}$, and
	for $M=\mathbb{S}^{d-1}$ the function $h$, see \eqref{Eq: Density of Q}, is given by
	\begin{align*}
	h: B^1_\varepsilon(0) \times \mathbb{S}^{d-1} &\to \mathbb{R_+} 
	\quad \text{with} \quad
	(t,u) \mapsto (1+t)^{d-1},
	\end{align*}
	which follows by the definitions in \cite{Weyl}.
	Hence, also \eqref{Eq: Assumption convergence happens from below} is satisfied and $W^p(\mu_n,\mu)$ converges for any $p\geq 1$ as $n^{-1/2}$ to zero.}
\end{example}

\subsection{Comments to related results} 
\label{Sec: Literature_review}

One of our goals is to quantify the weak convergence results formulated in \cite{Hwang}. 
Now we discuss alternative approaches that appear in the literature for achieving this.


For finitely many global minima and H\"older continuity of order $\alpha > 0$ of the function $\ell$ at the minimal points, 
Athreya and Hwang prove in \cite{Gibbsmeasuresasymptotics} convergence results for convergence in distribution of random variables $X_n \sim \mu_n$.
These results are of the form
\[
n^{1/\alpha} \left( X_n - x \right) \xrightarrow{d} Y,	
\]
where $x \in \mathbb{R}^d$ is a global minimal point of $\ell$, the arrow $\xrightarrow{d}$ denotes weak convergence and the distribution of the random variable $Y$ with exponential tails is centred at zero.
For $\alpha = 2$, our setting in Remark~\ref{R: Finitely many minima} can be treated with these results yielding a rate of $1/2$.
However, note that in general such a convergence in distribution as described above 
does not imply Wasserstein convergence of the distribution of $X_n$ to the Dirac measure $\delta_x$ at $x$. 
\textcolor{black}{In \cite{DegenerateMinima}, the results of \cite{Gibbsmeasuresasymptotics} are extended to the case of degenerate minima.}


Convergence rates in the $1$-Wasserstein distance in the case of finitely many global minima are provided by Bras and Pag\`{e}s in \cite{Bras}.
They discuss this convergence as an auxiliary result in the context of Langevin-based algorithms and likewise obtain a rate of $1/2$.
Their setting is very similar to what is assumed in Remark~\ref{R: Finitely many minima}, but is more restrictive, 
since an additional assumptions on the gradient and the Hessian of the function $\ell$ is required.


Finally, in \cite{Debortoli}, de Bortoli and Desolneux provide rates for the convergence in the $1$-Wasserstein distance if $\ell$ takes the form
\[
\ell(x) = \norm{F(x)}^k,
\]
where $k \in \mathbb{N}$ and {\color{black}$F: \mathbb{R}^d \to \mathbb{R}^q$, $q \in \mathbb{N}$}, is some well-behaved function. 
In particular, they allow the set of global minimal points of $\norm{F(x)}^k$ to be an arbitrary compact set.
For this setting, they obtain convergence rates of $1/k$ for the convergence to the limit measure.
In the case $k=2$, i.e., when we can apply our results to the setting in \cite{Debortoli}, 
this matches our rates obtained in both Theorem~\ref{Thm: Manifolds more general version} and Theorem~\ref{Thm: Manifolds convergence from below} for $p=1$.
De Bortoli and Desolneux observe that their results also hold for more general smooth functions $\ell: \mathbb{R}^d \to \mathbb{R}_+$ in the exponent
if the Hessian of $\ell$ is invertible on the set of global minimal points of $\ell$.
This matches our setting for the case $p=1$.
However, unlike de Bortoli and Desolneux, who essentially use Kantorovich-Rubinstein duality to obtain their results, 
we construct explicit couplings to obtain convergence rates.
This approach allows us to naturally derive results for arbitrary order $p\in\mathbb{N}$ of the $p$-Wasserstein distance.

\subsection{Proof of Theorem~\ref{Thm: Manifolds more general version}}
\label{Sec: proof_gen_thm}	

We start with stating auxiliary tools.
The following result is a consequence of an adapted Laplace method. It is proven in Section~\ref{Sec: Proof of psi-zeta-convergence lemma}.
\begin{lem}\label{L: Convergence psi, zeta}
	Suppose that the assumptions of Theorem~\ref{Thm: Manifolds more general version} hold. 
	Recall that $k_i := \dim M_i$ for $i \in \{1, \ldots, m\}$.
	Set
	\begin{align*}
	\psi^{(n)}_i(u) &:= \int_{B_\varepsilon^{d-k_i}(0)} \|t\|^p  \exp\big(-n\ell(S_i(t,u))\big) \pi_0\big(S_i(t,u)\big) h_i(t,u)
	\, \lambda_{d-k_i}(\textup{d}t), 
	\end{align*}
	and consider $\zeta^{(n)}_i$ as defined in \eqref{eq: zeta_i} for $n \in \mathbb{N}$ and $1\leq i \leq m$. 
	Then, for each $i \in \{1, \ldots, m\}$, there exists a constant $c_i > 0$ such that 
	\begin{equation} \label{eq: psi_conv}
	n^{(d-k_i)/2} \psi^{(n)}_i(u) \leq c_i \,n^{-p/2}, \qquad \forall\, u \in M_i.
	\end{equation}
	Moreover, there exists a constant $\widehat{c}_i > 0$, such that for all $u \in M_i$ holds
	\begin{equation} \label{eq: zeta_conv}
	\Bigg|\left(\frac{n}{2\pi}\right)^{(d-k_i)/2} \zeta^{(n)}_i(u) 
	- \pi_0 (S_i(0,u)) \det\left(\frac{\partial^2 (\ell \circ S_i)}{\partial^2 t} (0,u)\right)^{-1/2}\Bigg|
	\leq \widehat{c}_i\, n^{-1}.
	\end{equation}
\end{lem}	
\begin{rem}\label{R: Bound on normalizing constant for manifolds}
	In the setting of the previous lemma 
	for $i\in \{1,\dots,m\}$ set 
	\[
	\kappa_i:= \frac{(2 \pi)^{(d-k_i)/2}}{2} \min_{u \in M_i} \pi_0(S_i(0,u)) \det\Big(\frac{\partial^2 (\ell \circ S_i)}{\partial^2 t}(0,u)\Big)^{-1/2}>0.
	\]
	Then, \eqref{eq: zeta_conv} 
	implies that there exists an $n_0 \in \mathbb{N}$ such that for all $u \in M_i$ and all $n \geq n_0$ we have
	\begin{align} \label{al: low_bnd} 
	&n^{(d-k_i)/2} \zeta^{(n)}_i(u)\geq \kappa_i > 0.
	\end{align}
	For any $n\in\mathbb{N}$ define $a^{(n)}_i := n^{(d-k_i)/2} \int_{M_i} \zeta^{(n)}_i(u) \mathcal{M}_i(\d u)$ and note that \eqref{al: low_bnd} implies
	\begin{equation} 
	\label{low_bnd_ext1}
	n^{(d-k_i)/2} \int_{M_i} \zeta^{(n)}_i(u) \mathcal{M}_i(\d u) \geq \widehat{\kappa}_i, \quad n\in\mathbb{N},
	\end{equation}	
	with $\widehat{\kappa}_i = \min\{a^{(1)}_i,\dots,a^{(n_0-1)}_i,\mathcal{M}_i(M_i)\,\kappa_i \}>0$.
	Therefore, for $j \in \{1, \ldots, m\}$ with $\dim M_j = k = \max_{i =1 ,\ldots, m}k_i$, we have by \eqref{Eq: Measureconversion on tubular neighbourhood} that
	\begin{align*}
	n^{(d-k)/2} Z_n & =
	n^{(d-k)/2 }\int_{\mathbb{R}^d} \exp(-n\ell(x)) \pi_0(x) \, \lambda_d(\textup{d}x)\\
	& \geq n^{(d-k)/2} \int_{N_j(\varepsilon)} \exp(-n\ell(x)) \pi_0(x) \, \lambda_d(\textup{d}x) \\
	& =  n^{(d-k)/2} \int_{M_j} \zeta^{(n)}_j(u) \, \mathcal{M}_j( \textup{d}u)
	\geq  \widehat{\kappa}_j > 0, \quad \forall \, n\in\mathbb{N}.
	\end{align*}
	In particular, this means that $n^{(d-k)/2}Z_n$ is bounded from below by a positive constant that does not depend on $n$.
\end{rem}	
Now we formulate an estimate of the $p$-Wasserstein distance $W^p(\nu_n, \mu)$ with $\nu_n$ being defined as in \eqref{al: nu_n}. For the proof of this result we refer to Section~\ref{Sec: Convergence of intermediate step}.
\begin{lem}\label{L: Wassersteinconvergence of intermediate measure}
	Let the assumptions of Theorem~\ref{Thm: Manifolds more general version} 
	be satisfied and let $\nu_n$ be given as in \eqref{al: nu_n}.
	Then there exists a constant $K \in (0,\infty)$ such that
	\[
	W^p(\nu_n, \mu) \leq K n^{-1/(2p)}.
	\]
\end{lem}

The former result relies on a generic estimate of the $p$-Wasserstein distance which does not explicitly use a coupling construction. 
It exploits that $\nu_n$ is an approximation of $\mu$.
Now we deliver the coupling construction for estimating $W^p(\mu_n,\nu_n)$ which eventually gives the desired result.
\begin{proof}[Proof of Theorem \ref{Thm: Manifolds more general version}]
	We apply the strategy of proof outlined in Remark~\ref{rem: strategy_of_proof}. Note that from Lemma~\ref{L: Wassersteinconvergence of intermediate measure} we obtain an estimate of $W^p(\nu_n,\mu)$. Therefore, it is enough to prove a suitable bound of $W^{p}(\mu_n,\nu_n)$. To this end, we construct a coupling in three steps and estimate the resulting integrals. 
	Fix $i \in \{1, \ldots, m\}$ and $n \in \mathbb{N}$.
	
	\emph{1. Tubular neighbourhood part:}
	We start with constructing the part of the coupling of $\mu_n$ and $\nu_n$ that relates to the tubular neighbourhood $N_i(\varepsilon)$ of $M_i$.
	To this end, we use the measure $Q_i$ on $B^{d-k_i}_\varepsilon(0) \times M_i$ that is defined in \eqref{Eq: Density of Q}. 		
	Let $X_i$ and $Y_i$ be two random variables on $B_\varepsilon^{d-k_i}(0)$ respectively $M_i$ with joint distribution
	\[
	(X_i,Y_i) \sim \frac{1}{\lambda_d(N_i(\varepsilon))}Q_i.
	\]
	Obviously $\frac{1}{\lambda_d(N_i(\varepsilon))}Q_i$ is by \eqref{Eq: Normalizing constant of Q} a probability measure.
	Now we interpret $Z^{-1}_n \exp(-n\ell) \pi_0$ as a function on $N_i(\varepsilon) \times M_i$ mapping to $(0,\infty)$ via $(z,y) \mapsto Z^{-1}_n \exp\big(-n\ell(z)\big) \pi_0(z)$ and set
	\begin{equation}\label{Eq: Definition gamma_i,n}
	\gamma^{(n)}_i(\d z,\d y) := Z^{-1}_n \exp\big(-n\ell(z)\big) \pi_0(z) \cdot \mathbb{P}^{\left(S_i(X_i,Y_i), Y_i\right)}(\d z,\d y).
	\end{equation}
	Since $S_i$ is a diffeomorphism, 
	the set $\left\{S_i(A) \mid A\in \mathcal{B}\left(B_\varepsilon^{d-k_i}(0)\right) \otimes  \mathcal{B}(M_i)\right\}$ 
	is an intersection stable generator\footnote{\textcolor{black}{A collection of sets $\mathcal{E}$ is called an \emph{intersection stable generator} of a $\sigma$-algebra $\mathcal{F}$ if the smallest $\sigma$-algebra $ \sigma(\mathcal{E}) $ that contains $ \mathcal{E} $ satisfies $ \sigma(\mathcal{E}) = \mathcal{F} $, and for all $ E_1, E_2 \in \mathcal{E}  $ holds $ E_1 \cap E_2 \in \mathcal{E} $.}} of $\mathcal{B}\left(N_i(\varepsilon)\right)$.
	For $A\in \mathcal{B}\left(B_\varepsilon^{d-k_i}(0)\right) \otimes \mathcal{B}(M_i)$
	we have by \eqref{Eq: Imagemeasure of Q under T} and the bjectivity of $S_i$ that 
	\begin{align*}
	\mathbb{P}^{(S_i(X_i,Y_i))}\big(S_i(A)\big) & = \mathbb{P}\big((X_i,Y_i) \in A\big)= \frac{1}{\lambda_d(N_i(\varepsilon))}Q_i(A)\\
	&= \frac{1}{\lambda_d(N_i(\varepsilon))}(S_i)_*(Q_i)(S_i(A)) = \frac{1}{\lambda_d(N_i(\varepsilon))}\lambda_d\vert_{N_i(\varepsilon)}(S_i(A)),
	\end{align*}
	which yields
	\[
	\mathbb{P}^{(S_i(X_i,Y_i))} = \frac{1}{\lambda_d(N_i(\varepsilon))}\lambda_d\vert_{N_i(\varepsilon)}.
	\]
	Hence, for $A \in \mathcal{B}\left(N_i(\varepsilon)\right)$ we obtain
	\begin{align*}
	\gamma^{(n)}_i(A \times M_i) & = \int_A Z^{-1}_n \exp(-n\ell(z)) \pi_0(z) \, \mathbb{P}^{(S_i(X_i,Y_i))}(\textup{d}z)
	= \frac{\mu_n(A \cap N_i(\varepsilon))}{\lambda_d(N_i(\varepsilon))},
	\end{align*}
	and for $B \in \mathcal{B}(M_i)$ we have
	\begin{align*}
	&\gamma^{(n)}_i\big(N_i(\varepsilon) \times B\big) 
	= \int_{N_i(\varepsilon) \times M_i} \mathbbm{1}_B(y) Z^{-1}_n \exp(-n\ell(z)) \pi_0(z) \mathbb{P}^{(S_i(X_i,Y_i), Y_i)}(\textup{d}z, \textup{d}y)\\
	&\quad {\color{black}= \mathbb{E} \Big(\mathbbm{1}_B(Y_i)Z_n^{-1} \exp\big(-n \ell (S_i(X_i,Y_i))\big) \pi_0\big(S_i(X_i,Y_i)\big)\Big)}\\
	& \quad= \int_{B_\varepsilon^{d-k_i}(0) \times M_i} \mathbbm{1}_B(u) Z^{-1}_n \exp\big(-n\ell(S_i(t,u))\big) \pi_0\big(S_i(t,u)\big) 
	\, \mathbb{P}^{(X_i,Y_i)}(\textup{d}t, \textup{d}u)\\		
	&\quad = \int_B\int_{B_\varepsilon^{d-k_i}(0)} \frac{ \exp\big(-n\ell(S_i(t,u))\big) \pi_0\big(S_i(t,u)\big) h_i(t,u)}{Z_n\lambda_d(N_i(\varepsilon))}
	\, \lambda_{d-k_i}(\textup{d}t) \,  \mathcal{M}_i(\textup{d}u)\\
	&\quad =\frac{1}{\lambda_d(N_i(\varepsilon))} \int_B Z^{-1}_n \zeta^{(n)}_i(u) \,  \mathcal{M}_i(\textup{d}u),
	\end{align*}
	as $\color{black}\mathbb{P}^{(X_i,Y_i)} \sim \frac{Q_i}{\lambda_d(N_i(\varepsilon))}$.
	This means that the marginal measures of $\gamma^{(n)}_i$ are given by $\frac{1}{\lambda_d(N_i(\varepsilon))}\mu_n\vert_{N_i(\varepsilon)}$ and 
	$\frac{\zeta^{(n)}_i(u)}{Z_n\lambda_d(N_i(\varepsilon))}\cdot \mathcal{M}_i(\d u)$. \par
	Let $V_i(u)$ be the matrix with columns $v_i^{(1)}(u), \ldots, v_i^{(d-k_i)}(u)$ and 
	let $\psi^{(n)}_i(u)$ be as in Lemma~\ref{L: Convergence psi, zeta}.
	Then, we have
	\begin{align*}
	& \int_{N_i(\varepsilon) \times M_i} \|z-y\|^p \, \gamma^{(n)}_i(\textup{d}z, \textup{d}y)\\
	& =\int_{N_i(\varepsilon) \times M_i} Z^{-1}_n \|z-y\|^p  \exp(-n\ell(z)) \pi_0(z) \, \mathbb{P}^{(S_i(X_i,Y_i), Y_i)}(\textup{d}z, \textup{d}y) \\
	& =\int_{B_\varepsilon^{d-k_i}(0) \times M_i}Z^{-1}_n \|S_i(t,u)-u\|^p \exp\big(-n\ell(S_i(t,u))\big) \pi_0\big(S_i(t,u)\big) 
	\\
	&\hspace{10cm} \, 
	\times
	\mathbb{P}^{(X_i,Y_i)}(\textup{d}t, \textup{d}u)
	\\
	&= \frac{1}{Z_n\lambda_d(N_i(\varepsilon))} 
	\int_{M_i}  \int_{B_\varepsilon^{d-k_i}(0)} \|V_i(u)t\|^p \exp\big(-n\ell(S_i(t,u))\big) \pi_0\big(S_i(t,u)\big)  \\
	&\hspace{8.5cm} \times h_i(t,u)\, \lambda_{d-k_i}(\textup{d}t) \, \mathcal{M}_i(\textup{d}u) \\
	& \leq  \frac{1}{Z_n\lambda_d(N_i(\varepsilon))} \int_{M_i}   \|V_i(u)\|^p \psi^{(n)}_i(u) \, \mathcal{M}_i(\textup{d}u),
	\end{align*}
	where the third equality follows by the definition of $S_i$, see \eqref{al: def_S}.
	Note that $\|V_i(u)\|^p \equiv 1$ for all $u \in M_i$ as the columns of $V_i$ form an orthonormal basis.
	
	Therefore, using 
	\eqref{eq: psi_conv}
	and Remark \ref{R: Bound on normalizing constant for manifolds}, there exists a constant $K \in (0, \infty)$ such that
	\begin{align}\label{Eq: Convergence on the tubular neighbourhood}
	\begin{aligned}
	& \int_{N_i(\varepsilon) \times M_i} \|z-y\|^p \, \gamma^{(n)}_i(\textup{d}z, \textup{d}y)\\
	& \leq \frac{1}{\lambda_d(N_i(\varepsilon))} \cdot n^{(k-d)/2} Z^{-1}_n  \cdot n^{(k_i-k)/2} 
	\cdot \int_{M_i}  
	n^{(d-k_i)/2} \psi^{(n)}_i(u)
	\mathcal{M}_i(\textup{d}u)
	\leq K n^{-p/2}.
	\end{aligned}
	\end{align}
	
	\emph{2. Non-tubular neighbourhood part:}
	Next, we provide the part of the coupling that relates to $D := \mathbb{R}^d\setminus \bigcup_{i=1}^m N_i(\varepsilon)$.
	Note that $\sum_{i=1}^m \mu_n\big(N_i(\varepsilon)\big) < 1$.
	Therefore, we may define 
	the product measure
	\begin{align*}
	\xi_n & := \frac{1}{\mu_n(D)}  \mu_n\vert_D \otimes \beta_n
	\end{align*}
	where
	\begin{align*}
	\beta_n(\d u) := \left(\frac{1}{Z_n\sum_{j=1}^m \mu_n(N_j(\varepsilon)) }- Z^{-1}_n\right)\sum_{i=1}^m  \zeta^{(n)}_i(u) \cdot \mathcal{M}_i(\d u)
	\end{align*}
	is a measure on $R_\ell$.
	Note that by \eqref{Eq: Measureconversion on tubular neighbourhood} we have
	\begin{align}\label{Eq: Value beta_n(N)}
	\begin{aligned}
	& \beta_n(R_\ell) = \left(\frac{1}{\sum_{j=1}^m \mu_n(N_j(\varepsilon)) }- 1\right) \sum_{i=1}^m \int_{M_i} Z^{-1}_n \zeta^{(n)}_i(u) \, \mathcal{M}_i(\textup{d}u)   \\
	& =\left(\frac{1}{\sum_{j=1}^m \mu_n(N_j(\varepsilon)) }- 1\right)  \sum_{i=1}^m \mu_n\big(N_i(\varepsilon)\big) 
	= 1 - \sum_{i=1}^m \mu_n\big(N_i(\varepsilon)\big)
	= \mu_n(D).
	\end{aligned}
	\end{align}
	Moreover, by (T) from Assumption~\ref{ass: reg_to_ell} there exists a number $\delta > 0$ such that $\ell(z) \geq \delta$ for all $z \in D$. 
	Hence,
	\begin{align*}
	& \int_{\mathbb{R}^d \times R_\ell} \|z-y\|^p \, \xi_n(\textup{d}z, \textup{d}y) \\
	& \qquad= \frac{1}{\mu_n(D)}  \int_{R_\ell} \int_D Z^{-1}_n \exp(-n\ell(z)) \pi_0(z) \|z-y\|^p \, \lambda_d(\textup{d}z) \beta_n(\textup{d}y) \\
	& \qquad \leq \exp(-(n-1)\delta)n^{(d-k)/2}  \cdot n^{-(d-k)/2}Z^{-1}_n \\
	& \qquad \qquad  \cdot \frac{1}{\mu_n(D)}  \int_{R_\ell} \int_D  \exp(-\ell(z)) \pi_0(z) \|z-y\|^p \, \lambda_d(\textup{d}z) \beta_n(\textup{d}y).
	\end{align*}
		{\color{black}By the fact that 
		\[ 
		\norm{z-y}^p \leq 2^{p-1} \left(\norm{z-y_0}^p + \norm{y-y_0}^p\right), \quad  z,y \in \mathbb{R}^d,
		\] 
		we have for all $y\in \mathbb{R}^d$ that
		\begin{align*}
		\quad & \int_D  \exp(-\ell(z)) \|z-y\|^p \pi_0(z) \, \lambda_d(\textup{d}z)\\ & \leq 2^{p-1}\left(\int_{\mathbb{R}^d}  \exp(-\ell(z)) \|z-y_0\|^p \pi_0(z) \, \lambda_d(\textup{d}z) + \norm{y-y_0}^p Z_1\right) < \infty
		\end{align*}
		by Assumption (I).
		Moreover $y\mapsto\int_D  \exp(-\ell(z)) \|z-y\|^p \pi_0(z) \, \lambda_d(\textup{d}z)$ is continuous, and therefore bounded on $R_\ell$. 
	Additionally, we know by} \eqref{Eq: Value beta_n(N)} that $\frac{1}{\mu_n(D)}  \beta_n$ is a  probability measure, 
	such that there exists a constant $\kappa < \infty$ that satisfies
	\[
	\frac{1}{\mu_n(D)}  \int_{R_\ell} \int_D  \exp(-\ell(z)) \pi_0(z) \|z-y\|^p \, \lambda_d(\textup{d}z) \beta_n(\textup{d}y) \leq \kappa.
	\]
	Together with Remark \ref{R: Bound on normalizing constant for manifolds}, which implies that $n^{-(d-k)/2}Z_n^{-1}$ is bounded from above for all $n$, 
	there exists a constant $K \in (0,\infty)$ such that
	\begin{equation}\label{Eq: Convergence of the rest}
	\int_{\mathbb{R}^d \times R_\ell} \|z-y\|^p \, \xi_n(\textup{d}z, \textup{d}y) \leq K n^{-p/2}.
	\end{equation}
	
	\emph{3. Final coupling:}
	Now, we are able to provide a coupling $\eta_n\in C(\mu_n,\nu_n)$ by using $\gamma^{(n)}_i$ and $\xi_n$.
	For $n \in \mathbb{N}$ set
	\[
	\eta_n := \sum_{i=1}^m \lambda_d\big(N_i(\varepsilon)\big) \gamma^{(n)}_i + \xi_n.
	\]
	We check that it is indeed in $C(\mu_n,\nu_n)$. For $A \in \mathcal{B}(\mathbb{R}^d)$ we have 
	\begin{align*}
	\eta_n(A\times R_\ell) & = \sum_{i=1}^m \lambda_d(N_i(\varepsilon))\gamma^{(n)}_i(A \times M_i) + \xi_n(A \times R_\ell)\\
	& = \sum_{i=1}^m \lambda_d(N_i(\varepsilon)) \frac{\mu_n(A \cap N_i(\varepsilon))}{\lambda_d(N_i(\varepsilon))}
	+ \mu_n(A \cap D) \frac{\mu_n(D)}{\mu_n(D)} 
	= \mu_n(A),
	\end{align*}
	and for $B \in \mathcal{B}(R_\ell)$ we obtain
	\begin{align*}
	\eta_n(\mathbb{R}^d\times B) 
	& = \sum_{i=1}^m \lambda_d(N_i(\varepsilon)) \gamma^{(n)}_i\big(N_i(\varepsilon) \times B\big) + \xi_n(\mathbb{R}^d \times B)\\
	& = \sum_{i=1}^m \lambda_d(N_i(\varepsilon)) \frac{\int_B Z^{-1}_n \zeta^{(n)}_i(u) \, \mathcal{M}_i(\textup{d}u)}{\lambda_d(N_i(\varepsilon))}\\ 
	& \quad + \sum_{i=1}^m \int_B \left(\frac{}{Z_n\sum_{j=1}^m \mu_n(N_j(\varepsilon)) }- Z^{-1}_n\right)  \zeta^{(n)}_i(u)  \mathcal{M}_i(\textup{d}u))  \\
	&= \frac{1}{\sum_{j=1}^m \mu_n(N_j(\varepsilon)) }\sum_{i=1}^m \int_B Z^{-1}_n  \zeta^{(n)}_i(u) \mathcal{M}_i(\textup{d}u)) = \nu_n(B).
	\end{align*}
	Thus, eventually \eqref{Eq: Convergence on the tubular neighbourhood} and \eqref{Eq: Convergence of the rest} give
	\begin{align*}
	& \big(W^p(\mu_n, \nu_n)\big)^p  \leq \int_{\mathbb{R}^d \times R_\ell} \|z-y\|^p \, \eta_n(\textup{d}z, \textup{d}y)\\
	& =  \sum_{i=1}^m \lambda_d(N_i(\varepsilon)) \int_{N_i(\varepsilon)\times M_i} \|z-y\|^p \, \gamma^{(n)}_i(\textup{d}z, \textup{d}y) 
	+ \int_{\mathbb{R}^d \times R_\ell} \|z-y\|^p \, \xi_n(\textup{d}z, \textup{d}y) \\
	& \leq  \kappa n^{-p/2}
	\end{align*}
	for some constant $\kappa \in (0, \infty)$, which yields
	$
	W^p(\mu_n, \nu_n) \leq \widetilde{K} n^{-1/2}
	$
	for some constant $\widetilde{K} \in (0, \infty)$.
\end{proof}

\subsection{Proof of Theorem~\ref{Thm: Manifolds convergence from below}}
\label{Sec: proof_est_spec}
In contrast to the proof of Theorem~\ref{Thm: Manifolds more general version}, here we immediately construct a suitable coupling of $\mu_n$ and $\mu$. This is possible since by assumption we know that 
\begin{equation*}
\phi(u) \geq \sum_{i=1}^m Z_n^{-1} \zeta_{i}^{(n)}(u) \mathbbm{1}_{M_i}(u), \quad \forall\, u\in R_\ell,\, n\in\mathbb{N}.
\end{equation*}
After employing the coupling, we bound the resulting integrals by applying Lemma~\ref{L: Convergence psi, zeta} and its consequences.

\begin{proof}[Proof of Theorem \ref{Thm: Manifolds convergence from below}]
	
	We use (again) three steps to get the desired coupling and its integral properties. 
	Fix $i \in \{1, \ldots, m\}$ and $n\in\mathbb{N}$.
	
	\emph{1. Tubular neighbourhood part:}
	We provide the part of the coupling that relates to the tubular neighbourhood $N_i(\varepsilon)$.
	We use a normalized version of $\gamma^{(n)}_i$ from the proof of Theorem \ref{Thm: Manifolds more general version}, 
	which is defined in \eqref{Eq: Definition gamma_i,n}.
	Namely, set
	\[
	\widetilde{\gamma}^{(n)}_i : = \frac{1}{\gamma^{(n)}_i\left(N_i(\varepsilon) \times M_i\right) }\,\gamma^{(n)}_i.
	\] 
	Note that since the marginal measures of $\gamma^{(n)}_i$ are $\frac{1}{\lambda_d(N_i(\varepsilon))}\mu_n\vert_{N_i(\varepsilon)}$ 
	and $\frac{\zeta^{(n)}_i(u)}{Z_n\lambda_d(N_i(\varepsilon))}\cdot \mathcal{M}_i(\d u)$, we have
	\begin{align*}
	\gamma^{(n)}_i\big(N_i(\varepsilon) \times M_i\big) = \frac{1}{\lambda_d\big(N_i(\varepsilon)\big)}\mu_n\big(N_i(\varepsilon)\big).
	\end{align*}
	Hence, for $A \in \mathcal{B}\big(N_i(\varepsilon)\big)$ this gives
	\begin{align*}
	\widetilde{\gamma}^{(n)}_i(A\times M_i)  
	= \frac{\lambda_d\big(N_i(\varepsilon)\big)\mu_n\big(N_i(\varepsilon)\cap A \big)}{\mu_n\big(N_i(\varepsilon)\big)\lambda_d\big(N_i(\varepsilon)\big)} 
	= \frac{\mu_n\big(N_i(\varepsilon)\cap A\big)}{\mu_n\big(N_i(\varepsilon)\big)},
	\end{align*}
	and for $B \in \mathcal{B}(M_i)$ we obtain
	\begin{align*}
	\widetilde{\gamma}^{(n)}_i(N_i(\varepsilon)\times B)  
	& = \frac{\lambda_d\big(N_i(\varepsilon)\big)}{\mu_n\big(N_i(\varepsilon)\big)} \frac{\int_B Z^{-1}_n \zeta^{(n)}_i(u) 
		\,  \mathcal{M}_i(\textup{d}u)}{\lambda_d\big(N_i(\varepsilon)\big)}\\
	& = \frac{1}{\mu_n\big(N_i(\varepsilon)\big)} \int_B Z^{-1}_n  \zeta^{(n)}_i(u) \,  \mathcal{M}_i(\textup{d}u),
	\end{align*}
	such that the marginal measures of $\widetilde{\gamma}^{(n)}_i$ are $\frac{1}{\mu_n(N_i(\varepsilon))}\mu_n\vert_{N_i(\varepsilon)}$ 
	and $\frac{\zeta^{(n)}_i(u)}{Z_n\mu_n(N_i(\varepsilon))}\cdot \mathcal{M}_i(\d u)$.
	
	By \eqref{low_bnd_ext1} we have for $\widehat\kappa_i>0$ that
	\[
	n^{(d-k)/2} \int_{N_i(\varepsilon)} \exp(-n\ell(z)) \pi_0(z) \, \lambda_d(\textup{d}z)
	= n^{(d-k)/2} \int_{M_i} \zeta^{(n)}_i(u) \mathcal{M}_i(\d u) \geq \widehat{\kappa}_i.
	\]
	By \eqref{Eq: Convergence on the tubular neighbourhood} 
	this leads, 	for some constant $K \in (0,\infty)$,
	to
	\begin{align}\label{Eq: Manifolds convergence from below - convergence on tubular neighbourhood}
	\begin{aligned}
	& \int_{\mathbb{R}^d \times M_i} \|z-y\|^p \, \widetilde{\gamma}^{(n)}_i(\textup{d}z, \textup{d}y)
	\leq K n^{-p/2}.
	\end{aligned}
	\end{align}
	
	\emph{2. Non-tubular neighbourhood:}	
	Set $D := \mathbb{R}^d\setminus \big(\bigcup_{i=1}^m N_i(\varepsilon)\big)$ and define a measure
	\[
	\widetilde{\xi}_n({\rm d}z,{\rm d}y) :=  \frac{\mu_n\vert_D({\rm d}z)}{\mu_n(D)} \otimes \left( \frac{\phi(y) - \sum_{i=1}^m Z^{-1}_n \zeta^{(n)}_i(y) \mathbbm{1}_{M_i}(y) }
	{\int_{R_\ell} (\phi - \sum_{i=1}^m Z^{-1}_n \zeta^{(n)}_i \mathbbm{1}_{M_i})  \, \textup{d}\mathcal{M} } \cdot  \mathcal{M}(\d y) \right)
	\]
	on $D\times R_\ell$.
	By (T) of Assumption \ref{ass: reg_to_ell} there exists a $\delta > 0$ such that $\ell(z) \geq \delta$ for all $z \in D$ and we have
	\begin{align*}
	&\mu_n(D) \int_{\mathbb{R}^d \times R_\ell} \|z-y\|^p \, \widetilde{\xi}_n(\textup{d}z, \textup{d}y) \\
	&\quad \leq 
	\frac{\exp(-(n-1)\delta )n^{(d-k)/2} \cdot n^{-(d-k)/2}Z^{-1}_n }{\int_{R_\ell} ( \phi - \sum_{i=1}^m Z^{-1}_n \zeta^{(n)}_i \mathbbm{1}_{M_i}) \, \textup{d}\mathcal{M}}\\
	& \qquad \cdot \int_{R_\ell} \bigg(\phi(y) - \sum_{i=1}^m Z^{-1}_n \zeta^{(n)}_i (y) \mathbbm{1}_{M_i} (y)\bigg) 
	\int_{D} \|z-y\|^p  \exp(-\ell(z)) \pi_0(z)
	\lambda_d(\textup{d}z)\, \mathcal{M} (\textup{d}y).
	\end{align*}
	Note that $\int_{D} \|z-y\|^p  \exp(-\ell(z)) \pi_0(z) \, \lambda_d(\textup{d}z)$ is continuous in $y$ and therefore bounded on $R_\ell$.
	Moreover, $\frac{\phi(y) - \sum_{i=1}^m  Z^{-1}_n\zeta^{(n)}_i(y) \mathbbm{1}_{M_i}(y) }
	{\int_{R_\ell} (\phi - \sum_{i=1}^m Z^{-1}_n\zeta^{(n)}_i \mathbbm{1}_{M_i}) \, \textup{d}\mathcal{M}} \cdot  \mathcal{M}(\d y)$ 
	is a probability measure, and by Remark \ref{R: Bound on normalizing constant for manifolds}, $n^{-(d-k)/2}Z^{-1}_n$ is bounded from above (in $n$).
	Hence, there is a constant $K \in(0, \infty)$ such that
	\begin{equation}\label{Eq: Manifolds convergence from below - convergence of the rest}
	\mu_n(D) \int_{\mathbb{R}^d \times R_\ell} \|z-y\|^p \, \widetilde{\xi}_n(\textup{d}z, \textup{d}y) \leq K n^{-p/2}.
	\end{equation}
	
	\emph{3. Final coupling:} 
	Using $\widetilde{\gamma}^{(n)}_i$ and $\widetilde{\xi}_n$, we are able to provide a coupling $\widetilde{\eta}_n\in C(\mu_n,\mu)$.
	Define
	\begin{equation}\label{Eq: Rate providing coupling for convergence from below}
	\widetilde{\eta}_n := \sum_{i=1}^m \mu_n\big(N_i(\varepsilon)\big) \widetilde{\gamma}^{(n)}_i  + \mu_n(D) \widetilde{\xi}_n.
	\end{equation}
	Now we check that this is indeed such a claimed coupling. For this let
	$A \in \mathcal{B}(\mathbb{R}^d)$. Then
	\begin{align*}
	\widetilde{\eta}_n(A \times R_\ell) &= \sum_{i=1}^{m} \mu_n\big(N_i(\varepsilon)\big) \widetilde{\gamma}^{(n)}_i(A \times R_\ell)  + \mu_n(D) \widetilde{\xi}_n(A \times R_\ell) \\
	&= \sum_{i=1}^{m} \mu_n\big(N_i(\varepsilon)\big) \frac{\mu_n\big(A \cap N_i(\varepsilon)\big)}{\mu_n\big(N_i(\varepsilon)\big)} 
	+ \mu_n(D)  \frac{\mu_n(A \cap D)}{\mu_n(D)}  = \mu_n(A).
	\end{align*}
	Moreover, by \eqref{Eq: Measureconversion on tubular neighbourhood} we have
	\begin{align*}
	&\int_{R_\ell}(\phi(y)  - \sum_{i= 1}^m Z^{-1}_n\zeta^{(n)}_i(y) \mathbbm{1}_{M_i} (y)) \, \mathcal{M} (\textup{d}y) 
	= 1 - \sum_{i=1}^{m} \mu_n(N_i(\varepsilon)) = \mu_n(D).
	\end{align*}
	We thus obtain for $B \in \mathcal{B}(\mathbb{R}^d)$ that
	\begin{align*}
	\widetilde{\eta}_n\left(\mathbb{R}^d \times B\right) 
	&= \sum_{i=1}^{m} \mu_n\big(N_i(\varepsilon)\big) \widetilde{\gamma}^{(n)}_i\left(\mathbb{R}^d \times B\right)  
	+ \mu_n(D) \widetilde{\xi}_n\left(\mathbb{R}^d \times B\right) \\
	& =  \sum_{i=1}^{m} \mu_n\big(N_i(\varepsilon)\big) \frac{\int_{M_i \cap B} Z^{-1}_n\zeta^{(n)}_i(y)\, \mathcal{M}_i(\textup{d}y)}{\mu_n\big(N_i(\varepsilon)\big)} \\
	&\qquad+ \mu_n(D) \frac{\int_{R_\ell \cap B} (\phi(y) - \sum_{i=1}^m Z^{-1}_n\zeta^{(n)}_i(y) \mathbbm{1}_{M_i} (y))\, \mathcal{M}(\textup{d}y)}
	{\int_{R_\ell}(\phi(y) - \sum_{i=1}^m Z^{-1}_n\zeta^{(n)}_i(y) \mathbbm{1}_{M_i} (y)) \, \mathcal{M}(\textup{d}y)} \\
	& = \int_{R_\ell\cap B} \phi(y)\, \mathcal{M}(\textup{d}y)=  \mu(B).
	\end{align*}
	Therefore, $\widetilde{\eta}_n\in C(\mu_n,\mu)$ and by 
	\eqref{Eq: Manifolds convergence from below - convergence on tubular neighbourhood} and \eqref{Eq: Manifolds convergence from below - convergence of the rest} there exists some constant $K \in (0, \infty)$ such that
	\begin{align*}
	\Big(W^p(\mu_n, \mu)\Big)^{p} 
	&\leq \int_{\mathbb{R}^d \times R_\ell} \|z-y\|^p \, \widetilde{\eta}_n(\textup{d}z, \textup{d}y)\\
	& = \sum_{i=1}^{m} \mu_n(N_i(\varepsilon)) \int_{\mathbb{R}^d \times R_\ell} \|z-y\|^p \, \widetilde{\gamma}^{(n)}_i(\textup{d}z, \textup{d}y)\\
	&\qquad+  \mu_n(D) \int_{\mathbb{R}^d \times R_\ell} \|z-y\|^p \, \widetilde{\xi}_n(\textup{d}z, \textup{d}y)
	\leq K n^{-p/2}.
	\end{align*}
	Taking the $p$th root yields the statement of the theorem.
\end{proof}

\begin{appendix}
\section*{}
\label{Sec: Appendix}

{\color{black}	\subsection{Uniformity of modes: Proof of Lemma \ref{L: Uniformity of modes}}\label{Sec: Proof for uniformity of modes}
	By a change of variables we 
	verify the desired statement of the lemma.
	\begin{proof}[Proof of Lemma \ref{L: Uniformity of modes}]
	During the proof we use for $r>0$ the notation $B_r:=B^{d-k}_r(0)$. 
	By the fact that the matrices $ A_u $ depend continuously on $u\in R_{\ell}$, there exists $ \sigma>0 $ such that $ A_u^{-1}(B_{\widetilde{\varepsilon}}) \subseteq B_{\sigma}$ and
	w.l.o.g. for $\varepsilon$ from Assumption~\ref{ass: reg_to_ell} (T), (S) we may assume\footnote{Again by the continuity of $A_u$ regarding $u$, there is a number $\widetilde{\sigma}>0$, such that $B_{\widetilde{\sigma}} \subseteq A_u^{-1}(B_{\widetilde{\varepsilon}})$. In the case $\varepsilon>\widetilde{\sigma}$, we may newly choose $\varepsilon$ equal to $\widetilde{\sigma}$ and still satisfy for this smaller value Assumption~\ref{ass: reg_to_ell} (T), (S). In the other case $\varepsilon\leq \widetilde{\sigma}$ the subset property obviously holds.} that $B_\varepsilon \subseteq A_u^{-1}(B_{\widetilde{\varepsilon}})$. In total
		\begin{equation}\label{Eq: Smallest ball for affinity property}
	B_\varepsilon \subseteq  A_u^{-1}(B_{\widetilde{\varepsilon}}) \subseteq B_{\sigma},
	\end{equation} 
			for all $ u \in R_\ell $.
		Fix $ u \in M_i $ for some $ i \in \{1, \ldots,m\} $.
		By \eqref{Eq: Measureconversion on tubular neighbourhood}, \eqref{Eq: Definition of zeta} and \eqref{Eq: Smallest ball for affinity property} we have
		\begin{align*}
		&Z_n^{-1} \zeta_i^{(n)}(u) \leq \frac{\zeta_i^{(n)}(u)}{\sum_{j=1}^{m} \int_{N_j(\sigma)} \exp(-n(\ell(x)) \pi_0(x)\, \lambda_d(\textup{d}x)}\\
		&\quad= \frac{\int_{B_\varepsilon}   \exp(-n\ell(S_i(t,u))) \pi_0(S_i(t,u)) h_i(t,u)
			\, \lambda_{d-k}(\textup{d}t)}{\sum_{j=1}^{m}\int_{M_j}\int_{B_{\sigma}}   \exp(-n\ell(S_j(t,w))) \pi_0(S_j(t,w)) h_j(t,u)
			\, \lambda_{d-k}(\textup{d}t)\, \mathcal{M}_j(\textup{d}w)}\\
		& \quad\leq \frac{\int_{A_u^{-1}(B_{\widetilde{\varepsilon}})}   \exp(-n \ell(S_i(t,u))) \pi_0(S_i(t,u)) h_i(t,u)
			\, \lambda_{d-k}(\textup{d}t)}{\sum_{j=1}^{m}\int_{M_j}\int_{A_u^{-1}(B_{\widetilde{\varepsilon}})}   \exp(-n\ell(S_j(t,w))) \pi_0(S_j(t,w)) h_j(t,u)\,
			\lambda_{d-k}(\textup{d}t)\, \mathcal{M}_j(\textup{d}w)}.
		\end{align*}
		Using \eqref{Eq: Affinity conditions} yields
		\begin{align*}
		&Z_n^{-1} \zeta_i^{(n)}(u) \leq\\
		 &
	\frac{\int_{A_u^{-1}(B_{\widetilde{\varepsilon}})}   \exp(-n\ell(S_{j^*}(A_ut,v))) \pi_0(S_{j^*}(A_ut,v)) h_{j^*}(A_ut,v)
		\, \lambda_{d-k}(\textup{d}t)}
		{\sum_{j=1}^{m}\int_{M_j}\int_{A_u^{-1}(B_{\widetilde{\varepsilon}})}   
		\exp(-n\ell(S_{j^*}(A_wt,v)))  
					\pi_0(S_{j^*}(A_wt,v)) h_{j^*}(A_wt,v)
	\lambda_{d-k}(\textup{d}t)\, \mathcal{M}_j(\textup{d}w)}.
		\end{align*}
		By a change of variables, we obtain
		\begin{align*}
		&Z_n^{-1} \zeta_i^{(n)}(u)  \\
		&\leq \frac{\det(A_u)^{-1}}{\sum_{j=1}^{m}\int_{M_j} \det(A_w)^{-1}\, \mathcal{M}_j(\textup{d}w)} \frac{ \int_{B_{\widetilde{\varepsilon}}}   
			\exp(-n\ell( S_{j^*}(t,v))) \pi_0(S_{j^*}(t,v)) h_{j^*}(t,v)
		\, \lambda_{d-k}(\textup{d}t)}{
		\int_{B_{\widetilde{\varepsilon}}}  \exp(-n\ell(S_{j^*}(t,v)))  \pi_0(S_{j^*}(t,v)) h_{j^*}(t,v)
		\, \lambda_{d-k}(\textup{d}t)}\\
		&=\frac{\det(A_u)^{-1} }{\sum_{j=1}^{m}\int_{M_j} \det(A_w)^{-1}\, \mathcal{M}_j(\textup{d}w)}.
		\end{align*}
		Due to \eqref{Eq: Affinity conditions} and the multidimensional chain rule, we have
		\begin{align*}
		\det\Big(\frac{\partial^2 (\ell \circ S_i)}{\partial^2 t}(0,u)\Big)
		&= \det\Big(\frac{\partial^2\big( (\ell \circ S_{j^*})(A_u \cdot, v)\big)}{\partial^2 t}(0)\Big)\\
		&= \det\Big(\frac{\partial^2 (\ell \circ S_{j^*})}{\partial^2 t}(0,v)\Big) \det(A_u)^2.
		\end{align*} 
		Moreover \eqref{Eq: Affinity conditions} also implies
		\begin{align*}
		\pi_0(S_j(0,w)) = \pi_0(S_{j^*}(A_w0,v)) = \pi_0(S_{j^*}(0,v)) \qquad \forall w \in M_j,\, j \in \{1,\ldots,m\}.
		\end{align*}
		Consequently, we have 
		\begin{align*}
		Z_n^{-1} \zeta_i^{(n)}(u) 
		&\leq\frac{\det(A_u)^{-1} \det\Big(\frac{\partial^2 (\ell \circ S_{j^*})}{\partial^2 t}(0,v)\Big)^{-1/2}\pi_0(S_{j^*}(0,v))}{\sum_{j=1}^{m}\int_{M_j} \det(A_w)^{-1} \det\Big(\frac{\partial^2 (\ell \circ S_{j^*})}{\partial^2 t}(0,v)\Big)^{-1/2}\pi_0(S_{j^*}(0,v))\, \mathcal{M}_j(\textup{d}w)}\\
		&\leq\frac{\det\Big(\frac{\partial^2 (\ell \circ S_{i})}{\partial^2 t}(0,u)\Big)^{-1/2}\pi_0(S_{i}(0,u))   }{\sum_{j=1}^{m}\int_{M_j}  \det\Big(\frac{\partial^2 (\ell \circ S_{j})}{\partial^2 t}(0,w)\Big)^{-1/2}\pi_0(S_{j}(0,w))\, \mathcal{M}_j(\textup{d}w)} 
		 = \phi(u),
		\end{align*}
		and \eqref{Eq: Assumption convergence happens from below} is satisfied.
	\end{proof}}

\subsection{Adapted Laplace method: Proof of Lemma \ref{L: Convergence psi, zeta}}\label{Sec: Proof of psi-zeta-convergence lemma}
For the proof of Lemma~\ref{L: Convergence psi, zeta} we use a parametrized version of Laplace's method from \cite{Korshunov}.
For convenience of the reader we provide \cite[Theorem 2]{Korshunov} and add representations of coefficient functions that can be read off from their proof.

\begin{theorem}\label{Thm: Parametrized Laplace method}
	Let $U \subseteq \mathbb{R}^s$ be a compact set, $\Theta$ a compact parameter space and $r \in \mathbb{N}$.
	Let $f : U \times \Theta \to \mathbb{R}_+$ be a function which satisfies the following: \begin{itemize}
		\item \emph{Differentiability condition:}  The functions $f_\theta \colon U \to \mathbb{R}_+$ given by 
		$t \mapsto f(t, \theta)$ are $2r+2$-times continuously differentiable for all $\theta \in \Theta$.\\[-1.5ex]
		\item \emph{Uniformity condition:} All partial derivatives of $f_\theta$ up to order $2r + 2$ are continuous in $\theta$.\\[-1.5ex]
		\item \emph{Minimum condition:} For each $\theta \in \Theta$ the function $f_\theta$ has a unique minimal point $t_0(\theta)$ 
		with $f_\theta\left(t_0(\theta)\right) = 0$ which lies in the interior of $U$.\\[-1.5ex]
		\item \emph{Tail condition: } For all $\varepsilon > 0$ holds
		\[
		\inf\left \{ f(t, \theta) \mid \norm{t - t_0(\theta)} > \varepsilon, \theta \in \Theta \right\} > 0.				
		\]
		\item \emph{Positive definiteness condition:} The Hessian $\frac{\partial^2 f}{\partial^2 t}(t_0(\theta), \theta)$ of $f_\theta$ at $t_0(\theta)$ is positive definite
		for all $\theta \in \Theta$.
	\end{itemize}
	Moreover, let $g : U \times \Theta \to \mathbb{R}$ be a function such that 
	\[
	g_\theta\colon U\to \mathbb{R}, \quad t \mapsto g(t, \theta),
	\]
	is $2r$-times continuously differentiable for all $\theta \in \Theta$, 
	and all the partial derivatives of $g_\theta$ up to order $2r$ are continuous in $\theta$.
	
	Then, there exists a constant $c \in (0, \infty)$ and there are functions
	\[
	z^{(\alpha)}_{i,j} : \Theta \to \mathbb{R}, \qquad \forall i \in\{ 0, \ldots, r-1\},\;
	j\in\{0,\dots,2r\},\; \alpha\in \mathbb{N}_0^s,	
	\]
	such that
	\[
	\left  | n^{s/2} \int_U g(t,\theta) \exp\left( -n f(t, \theta) \right) \lambda_s(\d t) - \sum_{i = 0}^{r-1} c_i(\theta) n^{-i} \right| \leq c n^{-r}, \qquad \forall \theta\in \Theta,	
	\]
	with functions $c_i\colon \Theta \to \mathbb{R}$ for $i \in\{0, \ldots, r-1\}$ given by
	\[
	c_0(\theta) = (2 \pi)^{s/2} g\left( t_0(\theta), \theta \right) \det   \left(\frac{\partial^2 f}{\partial^2 t}(t_0(\theta), \theta) \right)^{-1/2}
	\]
	and 
	\[
	c_i(\theta) = \sum_{j = 0}^{2i} \sum_{\substack{\alpha \in \mathbb{N}_0^s\\ \sum_{m= 1}^s \alpha_m = j}} 
	D^{\alpha}g_\theta (t_0(\theta)) z_{i, j}^{(\alpha)}(\theta) , \qquad i \geq 1.
	\]
	(Here $D^{\alpha}g_\theta$ denotes the partial derivative of $g_\theta$ with respect to $\alpha \in \mathbb{N}^s_0$.)
\end{theorem}

\begin{rem}\label{R: Application of Laplace's method - p = 1}
	Since the function $c_0$ in Theorem \ref{Thm: Parametrized Laplace method} is given explicitly, the statement of the theorem for the case $r=1$ becomes
	\begin{align*}
	\left  | n^{s/2} \int_U g(t,\theta) \exp\left( -n f(t, \theta) \right) \lambda_s(\d t) 
	- (2 \pi)^{s/2} g\left( t_0(\theta), \theta \right)  \det\left( \frac{\partial^2 f}{\partial^2 t}(t_0(\theta), \theta) \right)^{-1/2} \right| 
	\leq \frac{c}{n}. 
	\end{align*}
\end{rem}

\begin{rem}\label{R: Application of Laplace's method - vanashing derivatives of g}
	Assume that the function $g$ in Theorem \ref{Thm: Parametrized Laplace method} additionally satisfies $g(t_0(\theta), \theta) = 0$ and
	\[
	D^\alpha g_\theta (t_0(\theta)) = 0, \qquad \forall \alpha \in \mathbb{N}_0^s \quad \text{with} \quad \sum_{m= 1}^s \alpha_m \leq 2r - 2,
	\]
	for all $\theta \in \Theta$.
	Clearly, then the functions $c_i$ for $i \in\{0, \ldots, r-1\}$ become zero and the statement of the theorem reads
	\[
	\left  | n^{s/2} \int_U g(t,\theta) \exp\left( -n f(t, \theta) \right) \lambda_s(\d t) \right| \leq c n^{-r}, \quad \forall \theta \in \Theta.
	\]
\end{rem}

\begin{proof}[Proof of Lemma~\ref{L: Convergence psi, zeta}]
	We fix an index $i \in \{0, \ldots, m\}$ and omit it for notational convenience at all manifold quantities.	
	We apply Theorem~\ref{Thm: Parametrized Laplace method} in the setting of  $s=d-k$, $U = B_\varepsilon^{d-k}(0)$, $\Theta = M$,
	\[
	f(t,u) = (\ell \circ S)(t,u), \qquad t \in B_\varepsilon^{d-k}(0), u \in M,
	\]
	and $t_0(u) = 0$ for all $u \in M$.
	
	For proving \eqref{eq: zeta_conv} we additionally choose $r = 1$ and    
	\[
	g(t,u) =  \pi_0\big(S(t,u)\big) h(t,u), \qquad t \in B_\varepsilon^{d-k}(0), u \in M.
	\]
	Then, by the fact that $h(0,u) = 1$ for all $u \in M$ (see \eqref{Eq: Density of Q}), 
	the inequality \eqref{eq: zeta_conv} follows by Remark~\ref{R: Application of Laplace's method - p = 1}.

	For proving \eqref{eq: psi_conv} we distinguish two cases:
	
	\emph{1. Case:} We assume that $p$ is even and additionally (to the setting for applying Theorem~\ref{Thm: Parametrized Laplace method}) we choose 
	$r = p/2$
	and 
	\[
	g(t,u) = \norm{t}^p \pi_0\big(S(t,u)\big) h(t,u), \qquad t \in B_\varepsilon^{d-k}(0), u \in M.
	\]
	Using \cite[Proposition 5]{Hardy} and the multinomial theorem, we obtain that all the partial derivatives up to order $2r -2 = p - 2 $ of $g_u\colon B_\varepsilon^{d-k}(0) \to \mathbb{R}$ (defined by $t \mapsto g(t,u)$) are zero
	at $t=0$ for all $u \in M$.
	Therefore, we are in the setting of Remark \ref{R: Application of Laplace's method - vanashing derivatives of g} 
	and Theorem \ref{Thm: Parametrized Laplace method} yields
	\begin{align*}
	n^{(d-k)/2} \psi^{(n)}(u) 
	&\leq c n^{-p/2}
	\end{align*}
	for some constant $c \in (0, \infty)$ and all $n \in \mathbb{N}$.
	
	\emph{2. Case:} We assume that $p$ is odd and consider $\zeta^{(n)}(u)$ as normalizing constant to apply H\"older's inequality for a $p$-norm w.r.t. a suitable probability measure. 
	We obtain 
	\begin{align*}
	\allowdisplaybreaks
	&n^{(d-k)/2} \psi^{(n)}(u) \\
	&= n^{(d-k)/2} \int_{B_\varepsilon^{d-k}(0)} \|t\|^p  \exp\big(-n\ell(S(t,u))\big) \pi_0\big(S(t,u)\big) h(t,u)
	\, \lambda_{d-k}(\d t)  \\
	&= n^{(d-k)/2} \zeta^{(n)}(u) \\
	&  \quad \Big( \Big( \int_{B_\varepsilon^{d-k}(0)} \|t\|^p  \frac{\exp\big(-n\ell(S(t,u))\big) 
		\pi_0\big(S(t,u)\big) h(t,u)}{\zeta^{(n)}(u)}  \, \lambda_{d-k}(\d t) \Big)^{1/p} \Big)^p\\	
	&\leq \left(n^{(d-k)/2} \zeta^{(n)}(u) \right)^{1/(p+1)} \\
	& \quad \Big( n^{(d-k)/2} \int_{B_\varepsilon^{d-k}(0)} \|t\|^{p+1}  
	\exp\big(-n\ell(S(t,u))\big) \pi_0\big(S(t,u)\big) h(t,u) \, \lambda_{d-k}(\d t) \Big)^{\frac{p}{p+1}}.
	\end{align*}
	The first factor can be bounded by a positive constant independently of $n \in \mathbb{N}$ and $u \in M$
	according to \eqref{eq: zeta_conv}, see also Remark~\ref{R: Bound on normalizing constant for manifolds}. The second factor, can be bounded by $c^{p/(p+1)} n^{-p/2}$ by applying the statement of the 1. case with the even number $p+1$. 
\end{proof}

\subsection{Convergence of $W^p(\nu_n, \mu)$: Proof of Lemma~\ref{L: Wassersteinconvergence of intermediate measure}} \label{Sec: Convergence of intermediate step}

For proving Lemma~\ref{L: Wassersteinconvergence of intermediate measure} we use 
an estimate of the $p$-Wasserstein distance formulated in the following auxiliary lemma. 


\begin{lem} \label{L: Coupling lemma for intermediate step}
	For compact $E\subseteq \mathbb{R}^d$ let $(E, \mathcal{E}, \rho)$ be a finite measure space. Let $g\colon E \to [0,\infty)$ and $g_n\colon E \to [0,\infty)$ for any $n\in\mathbb{N}$ be  probability density functions w.r.t. $\rho$ such that
	there exists a $K\in (0,\infty)$ and a $q>0$ with
	\begin{equation}  \label{eq: densities_close}
	\sup_{x \in E} |g_n(x) - g(x)| \leq Kn^{-q}, \quad  n\in\mathbb{N}.
	\end{equation}
	Define the
	probability measures $\tau(\d x):= g(x) \rho(\d x)$, $\tau_n(\d x):= g_n(x) \rho(\d x)$ 
	and	let $p \in \mathbb{N}$.
	Then, there exists a constant $K \in (0, \infty)$ such that
	\[
	W^p(\tau_n, \tau) \leq K n^{-q/p}, \quad \forall\, n\in\mathbb{N}.
	\] 
\end{lem}
\begin{proof}
	By the fact that $E$ is compact, there exists a number $\kappa\in(0,\infty)$ satisfying that $\sup_{x,y\in E} \Vert x-y \Vert \leq \kappa$. For all $x,y\in E$ this implies
	$
	\Vert x- y \Vert \leq \kappa \cdot \mathbbm{1}_{x\not=y}.
	$
	Then
	\begin{align*}
	W^p(\tau_n,\tau)^p & = \inf_{\gamma \in C(\tau_n, \tau)} \int_{E\times E} \|x-y\|^p \gamma(\textup{d}x, \textup{d}y)
	\leq \kappa^p  \inf_{\gamma \in C(\tau_n, \tau)} \int_{E \times E} \mathbbm{1}_{x\not=y}  \gamma(\textup{d}x, \textup{d}y)\\
	& = \kappa^p  \Vert \tau_n - \tau \Vert_{\rm tv} =
	\frac{1}{2} \kappa^p \int_E  \vert g_n(x)-g(x)\vert \rho(\d x) \leq \frac{1}{2} \kappa^p K \rho(E) n^{-q},
	\end{align*}
	where we have used two well-known identities concerning the total variation distance $\Vert \cdot \Vert_{\rm tv}$, see e.g. \cite[Lemma~3.4]{eberle}.
\end{proof}


The required estimate of the Wasserstein distance between $\nu_n$ and $\mu$ 
is a direct application of the previous lemma.

\begin{proof}[Proof of Lemma \ref{L: Wassersteinconvergence of intermediate measure}]
	
	We aim to apply Lemma~\ref{L: Coupling lemma for intermediate step} with $E=R_\ell$, $\rho=\sum_{i=1}^m \mathcal{M}_i$ and $\tau_n=\nu_n$, $\tau=\mu$, where
	\[
	g_n(u)=\frac{\sum_{i=1}^m \zeta^{(n)}_i(u) \mathbbm{1}_{M_i}(u)}{\sum_{i=1}^m \int_{M_i}\zeta^{(n)}_i(v) \, \mathcal{M}_i(\textup{d}v)}
	\qquad \text{and} \qquad
	g(u)= \begin{cases}
	\phi(u) & u\in R'\\
	0	& \text{otherwise}.
	\end{cases}
	\]
	With $1 \leq i \leq m$ and $u \in M_i$
	we define for abbreviating the notation
	\[
	q_i(u) := \pi_0(S_i(0,u)) \det\Big(\frac{\partial^2 (\ell \circ S_i)}{\partial^2 t}(0,u)\Big)^{-1/2}.
	\]
	By virtue of Lemma~\ref{L: Convergence psi, zeta}, by using similar arguments as in Remark~\ref{R: Bound on normalizing constant for manifolds}, there are constants $C_{i,1}$, $C_{i,2} \in (0,\infty)$ such that for all $u\in M_i$, $n\in\mathbb{N}$ we have
	\begin{equation}\label{Eq: Bounds on n^(d-k/2)zeta_n}
	C_{i,1}^{-1} \min_{u \in M_i} q_i(u) \leq \left(\frac{n}{2\pi}\right)^{(d-k_i)/2} \zeta_{i}^{(n)}(u) \leq C_{i,2} \max_{u \in M_i} q_i(u).
	\end{equation}
	Now, for verifying \eqref{eq: densities_close} we distinguish two cases:
	
	\textbf{1. Case:}
	For $1\leq i \leq m$ and some $M_i\subseteq R'$ we assume that $u\in M_i$. Then  
	\begin{align*}
	& \left|\frac{\sum_{j=1}^m \zeta^{(n)}_j(u) \mathbbm{1}_{M_j}(u)}{\sum_{j=1}^m \int_{M_j}\zeta^{(n)}_j(v) 
		\, \mathcal{M}_j(\textup{d}v)} - \phi(u)\right| \\
	&  = \Bigg|\frac{\left(\frac{n}{2\pi}\right)^{(d-k)/2}\zeta^{(n)}_i(u)\sum_{M_j\subseteq R'}\int_{M_j} q_j(v) 
		\, \mathcal{M}_j(\textup{d}v)}
	{(2\pi)^{\frac{k-d}{2}}\left(\sum_{j=1}^m \int_{M_j} n^{\frac{d-k}{2}}\zeta^{(n)}_j(v) \, \mathcal{M}_j(\textup{d}v)\right)
		\left(\sum_{M_j \subseteq R'}\int_{M_j} q_j(v) \, \mathcal{M}_j(\textup{d}v)\right)} \\
	&\quad - \frac{ q_i(u)\sum_{j=1}^m \int_{M_j} \left(\frac{n}{2\pi}\right)^{(d-k)/2}\zeta^{(n)}_j(v) \, \mathcal{M}_j(\textup{d}v)}
	{(2\pi)^{\frac{k-d}{2}}\left(\sum_{j=1}^m \int_{M_j} n^{\frac{d-k}{2}}\zeta^{(n)}_j(v) \, \mathcal{M}_j(\textup{d}v)\right)
		\left(\sum_{M_j\subseteq R'}\int_{M_j} q_j(v) \, \mathcal{M}_j(\textup{d}v)\right)} \Bigg|.
	\end{align*}
	Applying \eqref{low_bnd_ext1} in the second factor of the denominator to those terms with $k_j=k$ and neglecting the others leads to the fact that the denominator is bounded from below by a positive constant that is independent from $n$. 
	As $u \in M_i \subseteq R'$, for the numerator holds by Lemma \ref{L: Convergence psi, zeta}
	\begin{align*}
	&\left|\left(\frac{n}{2\pi}\right)^{\frac{d-k}{2}} \zeta^{(n)}_i(u)\sum_{M_j\subseteq R'}\int_{M_j} q_j \, \textup{d}\mathcal{M}_j
	- q_i(u)\sum_{j=1}^m \int_{M_j} \left(\frac{n}{2\pi}\right)^{\frac{d-k}{2}} \zeta^{(n)}_j \, \textup{d} \mathcal{M}_j\right|\\
	&\qquad \leq\left(\sum_{M_j\subseteq R'}\int_{M_j} q_j \, \textup{d}\mathcal{M}_j\right) 
	\left|\left(\frac{n}{2\pi}\right)^{\frac{d-k}{2}}\zeta^{(n)}_i(u) -q_i(u)\right|\\
	&\qquad \qquad + q_i(u)\left| \sum_{M_j \subseteq R'}\int_{M_j} q_j \, \textup{d}\mathcal{M}_j
	- \sum_{j=1}^m \int_{M_j}\left( \frac{n}{2\pi}\right)^{\frac{d-k}{2}}\zeta^{(n)}_j\, \textup{d}\mathcal{M}_j\right|\\
	&\qquad \leq\sum_{M_j \subseteq R'}\int_{M_j} q_j \, \textup{d}\mathcal{M}_j \cdot \widehat{c}_i n^{-1}
	+ q_i(u) \sum_{M_j \subseteq R'}\widehat{c}_j\mathcal{M}_j(M_j)  n^{-1}\\  
	&\qquad \qquad + q_i(u) \sum_{M_j\nsubseteq R'} C_{j,2} \max_{v \in M_j} q_j(v) \mathcal{M}_j(M_j) 
	\cdot  \left(\frac{n}{2\pi}\right)^{(k_j-k)/2}\\
	&\qquad \leq\left(\sum_{M_j \subseteq R'}\int_{M_j} q_j \, \textup{d}\mathcal{M}_j \widehat{c}_i 
	+\max_{v \in M_i} q_i(v) \sum_{M_j \subseteq R'}\widehat{c}_j\mathcal{M}_j(M_j) \right) n^{-1}  \\
	&\qquad \qquad	+ \max_{v \in M_i} q_i(v) \sum_{M_j\nsubseteq R'} C_{j,2} \max_{v \in M_j} p_j(v) \mathcal{M}_j(M_j)(2 \pi)^{(k-k_j)/2}
	\cdot n^{(k_j-k)/2}.
	\end{align*}
	
	\textbf{2. Case:} Assume that $u \in M_i \subseteq R_\ell\setminus R'$. Then
	by \eqref{Eq: Bounds on n^(d-k/2)zeta_n}
	we obtain
	\begin{align*}
	&\left|\frac{\sum_{j=1}^m \zeta^{(n)}_j(u) \mathbbm{1}_{M_j}(u)}{\sum_{j=1}^m \int_{M_j}\zeta^{(n)}_j(v)
		\, \mathcal{M}_j(\textup{d}v)} - \phi(u)\right| 
	= \Bigg|\frac{\left(\frac{n}{2\pi}\right)^{(d-k)/2}\zeta^{(n)}_i(u)}
	{\sum_{j=1}^m  \int_{M_j}\left(\frac{n}{2\pi}\right)^{(d-k)/2}\zeta^{(n)}_j(v)\, \mathcal{M}_j(\textup{d}v)} \Bigg|\\
	& \qquad\leq \Bigg|C_{j,1} C_{j,2} \frac{(2 \pi)^{-(k_i-k)/2}\max_{v \in M_i} q_i(v)}{\sum_{M_j\subseteq R'}\min_{v \in M_j} q_j(v) 
		\mathcal{M}_j(M_j)} \Bigg|n^{(k_i - k)/2}.
	\end{align*}
	
	Taking the cases together and using the fact that the remaining $k_j$ and $k_i$ are strictly smaller than $k$ leads to $n^{(k_j - k)/2} \leq n^{-1/2}$ and $n^{(k_i - k)/2} \leq n^{-1/2}$.
	Hence, there exists a constant $K\in(0,\infty)$ that is independent of $u \in R_\ell$ which satisfies
	\[
	\left| \frac{\sum_{j=1}^m \zeta^{(n)}_j(u) \mathbbm{1}_{M_j}(u)}{\sum_{j=1}^m \int_{M_j}\zeta^{(n)}_j(v) \, \mathcal{M}_j(\textup{d}v)}
	- \phi(u)\right| \leq K n^{-1/2}
	\]
	for all $u \in R_\ell$.
	Therefore, Lemma \ref{L: Coupling lemma for intermediate step} yields the desired estimate.
\end{proof}

\end{appendix}

%
%
%
%
%
\begin{acks}[Acknowledgments]
The authors thank Valentin de Bortoli, Tapio Helin, Remo Kretsch\-mann, Thomas M\"uller-Gronbach, Claudia Schillings and Philipp Wacker for fruitful discussions about this topic. 
\end{acks}
%
\begin{funding}
Mareike Hasenpflug gratefully acknowledges support of the DFG within pro\-ject 432680300 -- SFB 1456 subproject B02. 
\end{funding}

\end{document}